\documentclass[11pt]{amsart}%
\usepackage[latin9]{inputenc}
\usepackage{amsmath}
\usepackage{amssymb}
\usepackage{color}
\usepackage{amsfonts}
\usepackage{amsthm}
\usepackage{latexsym}
\usepackage{graphicx}
\usepackage{cite}%

\newtheorem{theorem}{Theorem}
\newtheorem{lemma}[theorem]{Lemma}
\newtheorem{proposition}[theorem]{Proposition}

\newtheorem{corollary}[theorem]{Corollary}

\definecolor{re}{rgb}{0.5,0.1,0}
\definecolor{blu}{rgb}{0,0.3,0.6}

\begin{document}
\title[Weighted Radon transforms of vector fields, with applications to MAET]{Weighted Radon transforms of vector fields, with applications to
magnetoacoustoelectric tomography}
\author{L Kunyansky$^1$, E McDugald$^1$ and B Shearer$^2$}

\address{$^1$ Department of Mathematics,
University of Arizona,
Tucson, AZ 85721,
United States of America
}

\address{$^2$ Department of Physics,
Cornell University,
Ithaca, NY 14853,
United States of America
}

\begin{abstract}
Currently, theory of ray transforms of vector and tensor fields
is well developed, but the Radon transforms of such fields have not been fully
analyzed. We thus consider linearly weighted and unweighted longitudinal and transversal
Radon transforms of vector fields. As usual, we use the standard Helmholtz decomposition of
smooth and fast decreasing vector fields over the whole space. We show that such a decomposition
produces potential and solenoidal components decreasing at infinity fast enough to
guarantee the existence of the unweighted longitudinal and transversal
Radon transforms of these components.

It is known that reconstruction of an arbitrary vector field from only longitudinal or
only transversal transforms is impossible. However, for the cases when both
linearly weighted and unweighted transforms of either one of the types are known,
we derive explicit inversion formulas for the full reconstruction of the field.
Our interest in the inversion of such transforms stems from a certain inverse
problem arising in magnetoacoustoelectric tomography (MAET). The connection
between the weighted Radon transforms and MAET is exhibited in the paper.
Finally, we demonstrate performance and noise sensitivity of the new inversion
formulas in numerical simulations.

\end{abstract}
\maketitle


\section{Introduction\label{S:vectorRadon}}

In this paper we study unweighted and linearly weighted Radon transforms of
vector fields. There is a significant body of work on ray transforms (that
involve integration over straight lines) of vector and tensor fields
\cite{norton-vector89,norton-vector92,strahlen-review,schuster-review,shar-book}%
. In particular, exponential and attenuated ray transforms were studied
in \cite{strahlen-expo,bukh-kaz-vector,natt-vector,bal-atten}, and momentum
ray transforms were investigated in \cite{shar-mom,mishra-weight}. However,
when it comes to the Radon transforms of vector fields (with integration over
hyperplanes), there are very few publications \cite{polya2015,polya2015num};
moreover, the consideration is usually restricted to unweighted transforms of
potential fields with finitely supported potentials. In the present paper we
consider general vector fields (i.e. not purely potential or solenoidal), and
we study both unweighted and linearly weighted Radon transforms.

As in the case of ray transforms, when studying the Radon transforms one
finds it convenient to use the Helmholtz decomposition. In other words, one splits a general vector field $F$ into the potential and
solenoidal parts $F^{\mathrm{p}}$ and $F^{\mathrm{s}},$ and considers
transversal and longitudinal Radon transforms of both $F^{\mathrm{p}}$ and
$F^{\mathrm{s}}$. However, even for a finitely supported field $F$ components
$F^{\mathrm{p}}$ and $F^{\mathrm{s}}$ are defined in the whole space
$\mathbb{R}^{d}$ and they are known to have only a polynomial decay at infinity. Thus,
in order to analyze the Radon transforms of $F^{\mathrm{p}}$ and
$F^{\mathrm{s}}$ one first needs to prove that such transforms do exist (i.e.
integrals over hyperplanes in $\mathbb{R}^{d}$ converge). This is not
completely trivial. In particular, the estimate given in the foundational book
\cite{shar-book} on ray transforms does not guarantee the convergence of the Radon transforms.
Thus, first we obtain an improved estimate for the rate of decay
at infinity of the potential and solenoidal parts $F^{\mathrm{p}}$ and
$F^{\mathrm{s}}$ of a fast decaying field $F.$ This estimate guarantees the
existence of the unweighted longitudinal and transversal Radon transforms of
$F^{\mathrm{p}}$ and $F^{\mathrm{s}}$.

Similarly to the case of ray transforms, the longitudinal Radon transforms
of a potential field vanish. The same is true for the transversal transform of
a solenoidal field. Therefore, reconstructing a general vector field from only
the longitudinal or only the transversal transform(s) is not possible.
However, it is not unusual in practice \cite{strahlen-expo} that one of the
transform types (either longitudinal or transversal) cannot be measured. In
order to replace missing information one may consider measuring weighted
transforms of the available type. For example, our interest in this problem
stems from a certain measurement scheme in the magnetoacoustoelectric
tomography (MAET). This scheme does not permit measuring a transversal
transform of a certain vector field, but, in addition to longitudinal
transforms one can measure linearly weighted longitudinal transforms of that field.

Below we present explicit formulas for solving two distinct problems. The
first problem is that of reconstructing a general vector field from known
values of its transversal transform, and from $d-1$ weighted transversal
transforms with various linear weights. The second problem (motivated by MAET)
is the reconstruction of a general vector field from $d-1$ of its longitudinal
transforms and one weighted longitudinal transform (again, with a linear
weight). The reader may want to compare our solutions of these problems
to the results of~\cite{mishra-weight}, where a full vector field is
reconstructed from a ray transform and a first-moment ray transform.

The rest of the paper is organized as follows. We define all the needed
transforms in Section~\ref{S:Formulation}.1 below, and we present
explicit solutions to the above two problems in Section~~\ref{S:Formulation}.2.
In Sections~\ref{S:trans} and~\ref{S:long} we provide proofs of the theorems
formulated in Section~\ref{S:Formulation}. Section~\ref{S:MAET} exhibits a potential application
of the Radon transforms of vector fields to a problem arising in MAET. We
further validate our theoretical results by numerical simulations,
see  Section \ref{S:NUMERICS}. Finally, the
proof of Theorem \ref{T:my_Helm} (on the rates of decay of $F^{\mathrm{p}}$
and $F^{\mathrm{s}}$) is relegated into the Appendix.

\section{Formulation of the main results\label{S:Formulation}}

\subsection{Definitions and technical estimates}

Consider a continuous function $f(x)$ defined in $\mathbb{R}^{d}$, subject to
the condition $f(x)=\mathcal{O}\left(  |x|^{-d}\right)  $ at infinity. Define
a hyperplane $\Pi(\omega,p)$ by the equation $\omega\cdot x=p,$ where
$\mathbb{S}^{d-1}$ is the unit sphere in $\mathbb{R}^{d}$, and ($\omega
,p)\in\mathbb{S}^{d-1}\times\mathbb{R}$. The Radon transform $\mathcal{R}f$ is
defined as the set of integrals of $f$ over all the hyperplanes:%
\[
\left[  \mathcal{R}f\right]  (\omega,p)\equiv\int\limits_{\Pi(\omega
,p)}f(x)\,dA_{\Pi}(x),\quad(\omega,p)\in\mathbb{S}^{d-1}\times\mathbb{R},
\]
where $dA_{\Pi}(x)$ is the standard area element on $\Pi(\omega,p).$
Properties of the Radon transform are traditionally studied for functions
$f(x)$ from the Schwartz class $\mathcal{S}(\mathbb{R}^{d})$. We recall that
this class consists of all $C^{\infty}(\mathbb{R}^{d})$ functions $f(x)$ whose
derivatives decay at infinity faster than any rational function:%

\begin{equation}
\sup_{x\in\mathbb{R}^{d}}|x^{\beta}D^{\alpha}f(x)|<\infty,\quad|\alpha
|=0,1,2,...,\quad|\beta|=0,1,2,..., \label{E:decay}%
\end{equation}
where $\alpha$ and $\beta$ are multiindeces, $\alpha=(\alpha_{1}%
,...,\alpha_{d}),$ $\beta=(\beta_{1},...,\beta_{d}),$ $\alpha_{j}$'s and
$\beta_{i}$'s are non-negative integers, $|\alpha|=\sum_{j=1}^{d}|\alpha
_{j}|,$ $|\beta|=\sum_{i=1}^{d}|\beta_{j}|,$ and
\[
D^{\alpha}f(x)=\frac{\partial^{|\alpha|}}{\partial^{\alpha_{1}}x_{1}%
\partial^{\alpha_{2}}x_{2}...\partial^{\alpha_{d}}x_{d}}f(x),\quad x^{\beta
}=x_{1}^{\beta_{1}}x_{2}^{\beta_{2}}...x_{d}^{\beta_{d}}.
\]
A function $f(x)\in\mathcal{S}(\mathbb{R}^{d})$ can be reconstructed from its
projections $g=\mathcal{R}f$ using the well known filtered backprojection
inversion formula~\cite{NattBook}:

\begin{equation}
f=\mathcal{R}^{-1}(g)\equiv\frac{1}{2}(2\pi)^{1-d}I^{-\alpha}\mathcal{R}%
^{\#}I^{\alpha-d+1}g, \label{E:FBP}%
\end{equation}
where $\mathcal{R}^{\#}$ is the dual Radon transform that acts on a function
$g(\omega,p)$ defined { on} $\mathbb{S}^{d-1}\times\mathbb{R}$ according to the
formula%
\[
\lbrack\mathcal{R}^{\#}g](x)=\int\limits_{\mathbb{S}^{d-1}}g(\omega
,\omega\cdot x)d\omega,
\]
and where the Riesz potential $I^{\alpha}f$ of a function $f$ is expressed
through the direct and inverse Fourier transforms $\mathcal{F}$ and
$\mathcal{F}^{-1}$ as follows%
\[
\lbrack I^{\alpha}f](x)=[\mathcal{F}^{-1}(|\xi|^{-\alpha}[\mathcal{F}%
f](\xi))](x).
\]

Let us consider now a continuous vector field $F(x)=(F_{1}(x),...,F_{d}(x))$
defined on $\mathbb{R}^{d},$ $d\geq2,$ whose $\,$components decay fast enough
for the existence of integrals over each hyperplane (e.g., $|F(x)|=\mathcal{O}%
\left(  |x|^{-d}\right)  $)$.$ Below we define several types of Radon
transforms of such a field.

The componentwise Radon transform $\mathfrak{R}F$ of $F$ is defined in the
obvious way:%
\[
\left[  \mathfrak{R}F\right]  (\omega,p)\equiv\left(  \mathcal{RF}%
_{1},..,\mathcal{RF}_{d}\right)  (\omega,p),\qquad(\omega,p)\in\mathbb{S}%
^{d-1}\times\mathbb{R}.
\]
The transversal Radon transform $\mathcal{D}^{\perp}F$ is the Radon transform
of the projection of $F$ onto the normal $\omega$ to the plane $\Pi(\omega
,p)$:%
\begin{equation}
\left[  \mathcal{D}^{\perp}F\right]  (\omega,p)\equiv\int\limits_{\Pi
(\omega,p)}\omega\cdot F(x)\,dA_{\Pi}(x)=\left[  \mathcal{R}(\omega\cdot
F(x))\right]  (\omega,p),\quad(\omega,p)\in\mathbb{S}^{d-1}\times\mathbb{R}.
\label{E:def_transversal}%
\end{equation}
For each fixed direction $\omega\in\mathbb{S}^{d-1}$, let us arbitrarily
extend $\omega$ to an orthonormal basis $\mathfrak{B}=(\omega,\omega
_{1},...,\omega_{d-1})$ of $\mathbb{R}^{d},$ where\ $\omega_{j}=\omega
_{j}(\omega),$ $j=1,...,d-1.$ To simplify the notation, below we will suppress
the dependence of $\omega_{j}$'s on$\ \omega.$ Define the longitudinal Radon
transforms $\mathcal{D}_{k}^{\shortparallel}F$ of $F,$ $k=1,...d-1$, as
follows:%
\begin{equation}
\left[  \mathcal{D}_{k}^{\shortparallel}F\right]  (\omega,p)\equiv
\int\limits_{\Pi(\omega,p)}\omega_{k}\cdot F(x)\,dA_{\Pi}(x)=\left[
\mathcal{R}(\omega_{k}\cdot F(x))\right]  (\omega,p),\quad(\omega
,p)\in\mathbb{S}^{d-1}\times\mathbb{R}. \label{E:def_longitudinal}%
\end{equation}
For a faster decaying vector field $F(x)$ (e.g. satisfying $|F(x)|=\mathcal{O}%
\left(  |x|^{-d-1}\right)  $), one can define the weighted transversal
transforms $\mathcal{W}_{k}^{\perp}$ and longitudinal transforms
$\mathcal{W}_{k}^{\shortparallel}$ with linear weights $\omega_{k}\cdot x,$
$k=1,...d-1,$ by the following expressions:%
\begin{align}
\left[  \mathcal{W}_{k}^{\perp}F\right]  (\omega,p)  &  \equiv\int%
\limits_{\Pi(\omega,p)}(\omega_{k}\cdot x)F(x)\cdot\omega\,dA_{\Pi}(x)=\left[
\mathcal{D}^{\perp}((\omega_{k}\cdot x)F(x))\right]  (\omega
,p),\label{E:def_Wk}\\
\left[  \mathcal{W}_{k}^{\shortparallel}F\right]  (\omega,p)  &  \equiv
\int\limits_{\Pi(\omega,p)}(\omega_{k}\cdot x)F(x)\cdot\omega_{k}\,dA_{\Pi
}(x)=\left[  \mathcal{D}_{k}^{\shortparallel}((\omega_{k}\cdot x)F(x))\right]
(\omega,p), \label{E:def_long_Wk}%
\end{align}
with $(\omega,p)\in\mathbb{S}^{d-1}\times\mathbb{R}$.

{ The present definitions of the unweighted longitudinal and transversal Radon transforms
coincide with those given in \cite{strahlen-expo,schuster-review} (where they
are mentioned under the names of ``probe'' and ``normal'' transforms, respectively).
Our definitions of the weighted transforms appear to be new; they naturally extend
the notion of ``moments ray transforms''~\cite{strahlen-expo,mishra-weight} to the case of Radon transforms.}

{
It is well known that the Radon transform of a scalar function considered on
$\mathbb{S}^{d-1}\times \mathbb{R}$ is redundant. Indeed, since $\Pi (\omega
,p)=\Pi (-\omega ,-p)$, one concludes that $\left[ \mathcal{R}f\right]
(\omega ,p)=\left[ \mathcal{R}f\right] (-\omega ,-p)$. Similarly, by
inspecting equation (\ref{E:def_transversal}) one can see that $\left[
\mathcal{D}^{\perp }F\right] (\omega ,p)=-\left[ \mathcal{D}^{\perp }F\right]
(-\omega ,-p)$, where the change of sign occurs due to the factor $\omega
\cdot $ under the integral. The definitions of transforms $\mathcal{D}%
_{k}^{\shortparallel }$, $\mathcal{W}_{k}^{\perp }$, and $\mathcal{W}%
_{k}^{\shortparallel }$ depend on two vectors, $\omega $ and $\omega _{k}$.
In general, our definition of basis $\mathfrak{B}$ permits a significant freedom in choosing the
dependence $\omega _{k}=\omega _{k}(\omega )$. However, if we restrict consideration
to the case $\omega
_{k}(\omega )=-\omega _{k}(-\omega )$, the following redundancies will arise%
\begin{eqnarray*}
\left[ \mathcal{D}_{k}^{\shortparallel }F\right] (\omega ,p) &=&-\left[
\mathcal{D}_{k}^{\shortparallel }F\right] (-\omega ,-p),\quad (\omega ,p)\in
\mathbb{S}^{d-1}\times \mathbb{R},\quad k=1,...d-1, \\
\left[ \mathcal{W}_{k}^{\perp }F\right] (\omega ,p) &=&\left[ \mathcal{W}%
_{k}^{\perp }F\right] (-\omega ,-p),\quad (\omega ,p)\in \mathbb{S}%
^{d-1}\times \mathbb{R},\quad k=1,...d-1, \\
\left[ \mathcal{W}_{k}^{\shortparallel }F\right] (\omega ,p) &=&\left[
\mathcal{W}_{k}^{\shortparallel }F\right] (-\omega ,-p),\quad (\omega ,p)\in
\mathbb{S}^{d-1}\times \mathbb{R},\quad k=1,...d-1.
\end{eqnarray*}%
Such redundancies can be exploited in practice, to reduce the number of
required measurements and to halve the number of floating point operations
when implementing inversion formulas, both known and the ones presented
below. (For example, operator $\mathcal{R}^{\#}$ in (\ref{E:FBP}) can be
computed by integration over a half of a sphere.) However, since the focus of
this paper is mostly theoretical, for simplicity of presentation we will
work with projections defined on $\mathbb{S}^{d-1}\times \mathbb{R}$.
}

For the future reference
we note the obvious relations%
\begin{equation}
\left[  \mathcal{D}^{\perp}F\right]  (\omega,p)=\omega\cdot\left[
\mathfrak{R}F\right]  (\omega,p),\qquad\left[  \mathcal{D}_{k}^{\shortparallel
}F\right]  (\omega,p)=\omega_{k}\cdot\left[  \mathfrak{R}F\right]
(\omega,p),\quad(\omega,p)\in\mathbb{S}^{d-1}\times\mathbb{R}.
\label{E:vectorRadon}%
\end{equation}

Let us now consider a smooth and fast decaying vector field $F(x)$ such that
each component $F_{m}(x)$ of $F(x)$ is a function from the Schwartz space
$\mathcal{S}(\mathbb{R}^{d})$. We define the potential $\varphi$ as the
convolution of the divergence $\Phi$ of $F$ with the fundamental solution $G$
of the Laplace equation in~$\mathbb{R}^{d}$:
\begin{equation}
\varphi(x)=(\Phi\ast G)(x)=\int\limits_{\mathbb{R}^{d}}\Phi(y)G(x-y)dy,\qquad
\Phi(x)=\operatorname{div}F(x),\qquad x\in\mathbb{R}^{d},
\label{E:potential_conv}%
\end{equation}
where explicit expressions for $G(x)\ $are well known:%
\[
G(x)=\frac{1}{2\pi}\ln|x|\text{ for }d=2,\qquad G(x)=-\frac{\Gamma
(d/2-1)}{4\pi^{2}}|x|^{2-d}\text{ for }d\geq3.
\]
Now the potential part $F^{\mathrm{p}}$ of the field $F$ is the gradient of
$\varphi$:%
\begin{equation}
F^{\mathrm{p}}(x)=\nabla\varphi(x),\qquad x\in\mathbb{R}^{d},
\label{E:def_pot_part}%
\end{equation}
and the solenoidal part $F^{\mathrm{s}}$ is just the difference%
\begin{equation}
F^{\mathrm{s}}(x)=F(x)-F^{\mathrm{p}}(x),\qquad x\in\mathbb{R}^{d}.
\label{E:def_sol_part}%
\end{equation}

The following theorem is a technical result that is an important tool in our investigation.

\begin{theorem}
Suppose that each component $F_{k}(x)$, $k=1,...,d$ of a vector field $F(x)$
is a function from the Schwartz class $\mathcal{S}(\mathbb{R}^{d}).$ Then
potential $\varphi$ and fields $F^{\mathrm{p}}$ and $F^{\mathrm{s}}$ given by
equations (\ref{E:potential_conv})-(\ref{E:def_sol_part}) have the following
decay rates at infinity\label{T:my_Helm}
\begin{align}
|\varphi(x)|  &  =\mathcal{O}\left(  \frac{1}{|x|^{d-1}}\right)
,\label{E:thm_phi}\\
|F^{\mathrm{p}}(x)|  &  =\mathcal{O}\left(  \frac{1}{|x|^{d}}\right)
,\qquad|F^{\mathrm{s}}(x)|=\mathcal{O}\left(  \frac{1}{|x|^{d}}\right)
,\label{E:thm_fields}\\
\left\vert \frac{\partial}{\partial x_{j}}F^{\mathrm{p}}(x)\right\vert  &
=\mathcal{O}\left(  \frac{1}{|x|^{d+1}}\right)  ,\qquad\left\vert
\frac{\partial}{\partial x_{j}}F^{\mathrm{s}}(x)\right\vert =\mathcal{O}%
\left(  \frac{1}{|x|^{d+1}}\right)  ,\label{E:thm_derivatives}\\
\left\vert \frac{\partial^{2}}{\partial x_{j}\partial x_{k}}F^{\mathrm{p}%
}(x)\right\vert  &  =\mathcal{O}\left(  \frac{1}{|x|^{d+2}}\right)
,\qquad\left\vert \frac{\partial^{2}}{\partial x_{j}\partial x_{k}%
}F^{\mathrm{s}}(x)\right\vert =\mathcal{O}\left(  \frac{1}{|x|^{d+2}}\right)
,\qquad j,k=1,2,...,d. \label{E:thm_second_der}%
\end{align}

\end{theorem}

The estimates (\ref{E:thm_phi})-(\ref{E:thm_second_der}) are a refinement of
the well known estimate on the rate of decay of $F^{\mathrm{p}}$ and
$F^{\mathrm{s}}$ given by Theorem 2.6.2 of \cite{shar-book}:
\begin{equation}
|F^{\mathrm{s}}(x)|\leq C(1+|x|)^{1-d}, \label{E:sharafut}%
\end{equation}
with the similar bound on $F^{\mathrm{p}}$. The importance of estimates
(\ref{E:thm_phi})-(\ref{E:thm_second_der}) for the present work is in that
they guarantee existence of the transversal, longitudinal, and component-wise
Radon transforms of $F^{\mathrm{p}}$ and $F^{\mathrm{s}},$ so that
\begin{equation}
\mathfrak{R}F=\mathfrak{R}F^{\mathrm{p}}+\mathfrak{R}F^{\mathrm{s}}%
,\quad\mathcal{D}^{\perp}F=\mathcal{D}^{\perp}F^{\mathrm{p}}+\mathcal{D}%
^{\perp}F^{\mathrm{s}},\quad\mathcal{D}_{k}^{\shortparallel}F=\mathcal{D}%
_{k}^{\shortparallel}F^{\mathrm{p}}+\mathcal{D}_{k}^{\shortparallel
} F^{\mathrm{s}}, \label{E:split}%
\end{equation}
with $k=1,2,...,d-1.$ Transforms $\mathcal{W}_{k}^{\shortparallel
}F^{\mathrm{p}}$, $\mathcal{W}_{k}^{\shortparallel}F^{\mathrm{s}}$,
$\mathcal{W}_{k}^{\perp}F^{\mathrm{p}}$, and $\mathcal{W}_{k}^{\perp
}F^{\mathrm{s}}$ cannot be defined, in general.
{ Indeed,
according to definitions~(\ref{E:def_Wk}) and ~(\ref{E:def_long_Wk}), such transforms would
require integration  of fields $F^{\mathrm{s}}$ and $F^{\mathrm{p}}$  multiplied
by linear functions in $x$, over hyperplanes in $\mathbb{R}^d$. Such products decay at infinity at the rate $\mathcal{O}(|x|^{1-d})$.
Such decay is not sufficient for the existence of the integrals. }

\subsection{Main theorems}

The main results of this paper are the following two theorems:

\begin{theorem}
If an infinitely differentiable vector field $F(x)=(F_{1}(x),...,F_{d}(x))$
satisfies decay conditions (\ref{E:decay}), its divergence $\Phi$ can be
reconstructed from the transversal transform $\mathcal{D}^{\perp}F$ by
applying the inversion formula (\ref{E:FBP}) as follows%
\begin{equation}
\Phi(x)=\left[  \mathcal{R}^{-1}\left(  \frac{\partial}{\partial p}%
\mathcal{D}^{\perp}F\right)  \right]  (x),\qquad x\in\mathbb{R}^{d}.
\label{E:T_trans_1}%
\end{equation}
Further, the componentwise Radon transform of $F$ can be reconstructed from
$\mathcal{D}^{\perp}F$ and\ weighted transversal transforms with linear
weights $\mathcal{W}_{k}^{\perp}F$, $k=1,..,d-1,$ as follows:\label{T:trans}%
\begin{equation}
\lbrack\mathfrak{R}F](\omega,p)=\omega\lbrack\mathcal{D}^{\perp}%
F](\omega,p)+\sum_{k=1}^{d-1}\omega_{k}\left(  \frac{\partial}{\partial
p}[\mathcal{W}_{k}^{\perp}F](\omega,p)-[\mathcal{R}\{(\omega_{k}\cdot
x)\Phi(x)\}](\omega,p)\right)  \label{E:T_trans_2}%
\end{equation}
where $(\omega,p)\in\mathbb{S}^{d-1}\times\mathbb{R}$, $j=1,2,...,d.$ Finally,
field $F$ can be recovered by inverting $\mathfrak{R}F$ componentwise:%
\begin{equation}
F_{j}(x)=\mathcal{R}^{-1}\left(  e_{j}\cdot\mathfrak{R}F\right)  (x),\qquad
x\in\mathbb{R}^{d},\qquad j=1,2,...,d, \label{E:T_trans_3}%
\end{equation}
where vectors $e_{1,}e_{2},...,e_{d}$ form the canonical orthonormal basis in
$\mathbb{R}^{d}$, and where $\mathcal{R}^{-1}$ is understood as the
filtration/backprojection formula (\ref{E:FBP}).
\end{theorem}

In order to formulate the next theorem, let us denote by $\Psi$ the
componentwise Laplacian $\Psi$ of the solenoidal part of the field
$F^{\mathrm{s}}$:%
\[
\Psi(x)\equiv(\Psi_{1}(x),\Psi_{2}(x),...,\Psi_{{ d} }(x)),\quad\Psi
_{j}(x)=\Delta F_{j}^{s}(x),\quad x\in\mathbb{R}^{d},\quad j=1,..,d.
\]

\begin{theorem}
If an infinitely differentiable vector field $F(x)=(F_{1}(x),...,F_{d}(x))$
satisfies decay conditions (\ref{E:decay}), the componentwise Laplacian $\Psi$
of its solenoidal part $F^{\mathrm{s}}$ and the Radon transform of $\Psi$
\ can be reconstructed from longitudinal transforms $\mathcal{D}_j%
^{\shortparallel}F$, $j=1,...,d-1,$ using the following
formulas\label{T:long}:%
\begin{align}
\lbrack\mathfrak{R}\Psi](\omega,p)  &  =\frac{\partial^{2}}{\partial p^{2}%
}\sum\limits_{j=1}^{d-1}\omega_{j}[{\mathcal{D}_j}^{\shortparallel
}F](\omega,p),\nonumber\\
\Psi_{j}(x)  &  =\left[  \mathcal{R}^{-1}\left(  e_{j}\cdot\mathfrak{R}%
\Psi\right)  \right]  (x),\qquad x\in\mathbb{R}^{d},\qquad j=1,2,...,d.
\label{E:invert_for_sol}%
\end{align}
Further, the divergence $\Phi$ of the field can be reconstructed from the
linearly weighted longitudinal transform $\mathcal{W}_{1}^{\shortparallel}F$
and previously found $\Psi$ as follows:%
\begin{equation}
\Phi(x)=\mathcal{R}^{-1}\{\mathcal{R(}(x\cdot\omega_{1})\omega_{1}\cdot
\Psi(x))-\frac{\partial^{2}}{\partial p^{2}}\mathcal{W}_{1}^{\shortparallel
}F\},\qquad x\in\mathbb{R}^{d}, \label{E:invert_for_pot}%
\end{equation}
where $\mathcal{R}^{-1}$ is understood as the filtration/backprojection
formula (\ref{E:FBP}). Finally, filed $F$ is reconstructed from $\Phi$ and
$\Psi$ by convolving these functions with $G$ and its gradient:%
\begin{equation}
F(x)=(\Phi\ast\nabla G)(x)+ { \sum_{j=1}^{d} e_j }(\Psi_{j}\ast G)(x), \quad
x\in\mathbb{R}^{d}. \label{E:convolve_sol}%
\end{equation}

\end{theorem}

We provide the proofs of theorems (\ref{T:trans}) and (\ref{T:long}) in Sections~\ref{S:trans}
and~\ref{S:long}, respectively. The proof of theorem
(\ref{T:my_Helm}) can be found in the Appendix.\

\section{Properties of the transversal transforms and \label{S:trans}proof of
Theorem \ref{T:trans}}

\subsection{Reconstructing the potential part of the
field\label{S:rec_pot_part}}

Most of the material reviewed in the present section \ref{S:rec_pot_part} \ is
known. However, to make the presentation self-contained, we provide elementary proofs.

\begin{proposition}
Suppose $F^{\mathrm{s}}(x)$ is a differentiable solenoidal vector field
decreasing at infinity at the rate $F^{\mathrm{s}}(x)=\mathcal{O}\left(
|x|^{-d}\right)  .$ Then the transversal Radon transform $\mathcal{D}^{\perp
}F^{\mathrm{s}}$ of $F^{\mathrm{s}}$ vanishes\label{T:prop_solenoid}:%
\begin{equation}
\lbrack\mathcal{D}^{\perp}F^{\mathrm{s}}](\omega,p)=\omega\cdot\left[
\mathfrak{R}F^{\mathrm{s}}\right]  (\omega,p)=0,\qquad(\omega,p)\in
\mathbb{S}^{d-1}\times\mathbb{R}. \label{E:solenoid_zero}%
\end{equation}

\begin{proof}
Fix an arbitrary pair $(\omega,p)\in\mathbb{S}^{d-1}\times\mathbb{R}$ and the
corresponding hyperplane $\Pi(\omega,p).$ Consider a sphere $S(0,R)$ of radius
$R$ centered at the origin. Further, consider the region $\Upsilon(R,p)$
bounded by a part of $S(0,R)$ and $\Pi(\omega,p),$ and such that the interior
normal to the boundary of $\Upsilon(R,p)$ on $\Pi(\omega,p)$ coincides with
$\omega$.

{ Let us denote by $\partial\Upsilon_{1}(R,p)$ the spherical part of the boundary
$\Upsilon(R,p)$, i.e.
 $\partial\Upsilon_{1}(R,p)\equiv\partial\Upsilon(R,p)\cap S(0,R)$}.
 Since
$\operatorname{div}F^{\mathrm{s}}=0,$ the following integrals are equal%
\[
\int\limits_{B(0,R)\cap\Pi(\omega,p)}F^{\mathrm{s}}(x)\cdot\omega\,dA_{\Pi
}(x)=\int\limits_{\partial\Upsilon_{1}(R,p)}F^{\mathrm{s}}(x)\cdot
n(x)\,dS(x)
\]
where $dS(x)$ is the standard area element on $S(0,R)$ and $n(x)$ is the
exterior normal to the sphere. Now, let us take the limit $R\rightarrow
\infty.$ Due to the fast decrease of $F^{\mathrm{s}}(x)$ at infinity, the
right hand side in the above equation converges to 0. The left hand side
converges to $\int\nolimits_{\Pi(\omega,p)}F(x)\cdot\omega\,dA_{\Pi}(x),$
proving that this integral is equal to 0. Since this is true for arbitrary
$(\omega,p),$ equation (\ref{E:solenoid_zero}) follows.
\end{proof}
\end{proposition}

The corollary below follows immediately from Proposition \ref{T:prop_solenoid}
and Theorem \ref{T:my_Helm}.

\begin{corollary}
Suppose $F(x)$ is a $C^{\infty}$ vector field defined on $\mathbb{R}^{d}$ and
decaying at infinity at rates given by equation (\ref{E:decay}), and
$F^{\mathrm{p}}+F^{\mathrm{s}}$ are defined by equations
(\ref{E:potential_conv})-(\ref{E:def_sol_part}). Then%
\begin{equation}
\mathcal{D}^{\perp}F=\mathcal{D}^{\perp}(F^{\mathrm{p}}+F^{\mathrm{s}%
})=\mathcal{D}^{\perp}F^{\mathrm{p}}. \label{E:prop_potential}%
\end{equation}

\end{corollary}

Suppose $h$ is a Radon integrable
function with a Radon integrable derivative $\frac{\partial}{\partial
x_{k}}h$. Then the following relation holds~\cite{helgason2013Radon}:
\[
\mathcal{R}\left[  \frac{\partial}{\partial x_{k}}h\right]  \mathcal{=}%
(e_{k}\cdot\omega)\frac{\partial}{\partial p}\mathcal{R}h.
\]
This leads to the following Lemma.

\begin{lemma}
\label{T:Radon_of_div}Suppose vector field $H(x)=(H_{1}(x),...,H_{d}(x))$ is
differentiable and decays at infinity at the rate $\left|H(x)\right|=\mathcal{O}\left(
|x|^{-d}\right)  $ or faster, with $\frac{\partial H_{k}}{\partial x_{k}%
}=\mathcal{O}\left(  |x|^{-d}\right)$,  $k=1,...,d$. Then%
\[
\left[  \mathcal{R(}\operatorname{div}H)\right]  (\omega,p)=\frac{\partial
}{\partial p}\left[  \mathcal{D}^{\perp}H\right]  (\omega,p),\qquad
(\omega,p)\in\mathbb{S}^{d-1}\times\mathbb{R}.
\]

\begin{proof}
The divergence $\operatorname{div}H(x)$ has the rate of decay $\mathcal{O}%
\left(  |x|^{-d}\right)  $, justifying the following:
\begin{align*}
\mathcal{R(}\operatorname{div}H)  &  =\sum\limits_{k=1}^{d}\mathcal{R}\left(
\frac{\partial H_{k}}{\partial x_{k}}\right)  =\sum\limits_{k=1}^{d}%
(e_{k}\cdot\omega)\frac{\partial}{\partial p}\mathcal{R}H_{k}=\frac{\partial
}{\partial p}\mathcal{R}\sum\limits_{k=1}^{d}(e_{k}\cdot\omega)H_{k}\\
&  =\frac{\partial}{\partial p}\mathcal{R}(\omega\cdot H)=\frac{\partial
}{\partial p}\left(  \omega\cdot\mathfrak{R}(H)\right)  =\frac{\partial
}{\partial p}\mathcal{D}^{\perp}H,
\end{align*}
where equation (\ref{E:vectorRadon}) is used on the second line of equalities.
\end{proof}
\end{lemma}

In particular for a $C^{\infty}$ field $F$ satisfying the rates of decay
(\ref{E:decay}) we obtain%
\begin{equation}
\mathcal{R}\Phi=\frac{\partial}{\partial p}\mathcal{D}^{\perp}F.
\label{E:useful_formula}%
\end{equation}
Since $\Phi$ is a function from the Schwartz class $\mathcal{S}(\mathbb{R}%
^{d})$ it can be reconstructed from projections using the filtered
backprojection formula (\ref{E:FBP}), which yields equation (\ref{E:T_trans_1}%
). The potential part of the field $F^{\mathrm{p}}(x)$ can now be computed by
combining (\ref{E:potential_conv}) and (\ref{E:def_pot_part}):%
\begin{equation}
F^{\mathrm{p}}(x)=\nabla\left(  G\ast\Phi\right)  (x)=(\Phi\ast\nabla G)(x).
\label{E:find_potential_part}%
\end{equation}

\subsection{Reconstructing the whole filed}

Due to Proposition \ref{T:prop_solenoid}, the solenoidal part $F^{\mathrm{s}}$
of the field $F$ lies in the null space of the transversal Radon transform
$\mathcal{D}^{\perp}$, and therefore, cannot be reconstructed from the
knowledge of $\mathcal{D}^{\perp}F.$ Thus, in addition to $\mathcal{D}^{\perp
}F$, in this section we assume the knowledge of the transversal weighted
transforms $\mathcal{W}_{k}^{\perp}F,$ $k=1,...,d-1,$ defined by
(\ref{E:def_Wk}). This information will allow us to reconstruct the whole
field $F$ and thus to complete the proof of theorem \ref{T:trans}

First, for the future use we would like to find projections of $\mathfrak{R}F$ on the
vectors of the basis $\mathfrak{B.}$ By combining equations
(\ref{E:vectorRadon}) and (\ref{E:prop_potential}) one observes:%
\begin{equation}
\omega\cdot\left[  \mathfrak{R}F\right]  (\omega,p)=\left[  \mathcal{D}%
^{\perp}F^{\mathrm{p}}\right]  (\omega,p),\qquad(\omega,p)\in\mathbb{S}%
^{d-1}\times\mathbb{R}. \label{E:proj_to_omega}%
\end{equation}
Let us find projections of $\mathfrak{R}F$ on vectors $\omega_{1}%
,...,\omega_{d-1}$ of the basis $\mathfrak{B.}$ Note that, due to
(\ref{E:def_longitudinal})
\[
\omega_{k}\cdot\left[  \mathfrak{R}F\right]  =\mathcal{D}_{k}^{\shortparallel
}F,,\qquad k=1,..,d-1.
\]
We start with $\omega_{1}\cdot\mathfrak{R}F^{\mathrm{p}}:$%

\begin{align*}
\omega_{1}\cdot\lbrack\mathfrak{R}F^{\mathrm{p}}](\omega,p)  &  =[\mathcal{D}%
_{1}^{\shortparallel}F^{\mathrm{p}}](\omega,p)=\int\limits_{\Pi(\omega
,p)}\omega_{1}\cdot F^{\mathrm{p}}(x)\,dA_{\Pi}(x)\\
&  =\int\limits_{\mathbb{R}}...\int\limits_{\mathbb{R}}\left[
\int\limits_{\mathbb{R}}\omega_{1}\cdot F^{\mathrm{p}}(p\omega+y_{1}%
\omega_{1}+...+y_{d-1}\omega_{d-1})\,dy_{1}\right]  \,dy_{2}...\,dy_{d-1}\\
&  =\int\limits_{\mathbb{R}}...\int\limits_{\mathbb{R}}\left[
\int\limits_{\mathbb{R}}\frac{\partial}{\partial y_{1}}\varphi
(p\omega+y_{1}\omega_{1}+...+y_{d-1}\omega_{d-1})\,dy_{1}\right]
\,dy_{2}...\,dy_{d-1}\\
&  =\int\limits_{\mathbb{R}}...\int\limits_{\mathbb{R}}\left[
 \lim_{\substack{a\rightarrow+\infty\\b\rightarrow
-\infty}}\left.  \varphi(p\omega+y_{1}\omega_{1}+...+y_{d-1}\omega
_{d-1})\right\vert _{y_{1}=b}^{y_{1}=a}\right]  \,dy_{2}\,...\,dy_{d-1}=0,
\end{align*}
for any $(\omega,p)\in\mathbb{S}^{d-1}\times\mathbb{R}.$ Since the numbering
of vectors $\omega_{1},...,\omega_{d-1}$ in the basis $\mathfrak{B}$ is
arbitrary, we conclude that%
\begin{equation}
\omega_{k}\cdot\lbrack\mathfrak{R}F^{\mathrm{p}}](\omega,p)=[\mathcal{D}%
_{k}^{\shortparallel}F^{\mathrm{p}}](\omega,p)=0,\qquad k=1,...,d-1,\qquad
(\omega,p)\in\mathbb{S}^{d-1}\times\mathbb{R}. \label{E:orthogonality}%
\end{equation}
In other words, a longitudinal transform of a potential field vanishes. Since
basis $\mathfrak{B}$ is orthonormal, by combining (\ref{E:orthogonality}) with
(\ref{E:proj_to_omega}) one obtains the following formula:
\begin{equation}
\mathfrak{R}F^{\mathrm{p}}=\omega\left\langle \omega\cdot\mathfrak{R}%
F^{\mathrm{p}}\right\rangle =\omega\mathcal{D}^{\perp}F^{\mathrm{p}}%
=\omega\mathcal{D}^{\perp}F. \label{E:pot_part}%
\end{equation}
Thus, the componentwise Radon transform of the potential part $F^{\mathrm{p}}$
of the field $F$ can be easily recovered from the transversal transform
$\omega\mathcal{D}^{\perp}F$.

Let us find what information can be extracted from the weighted transversal
transforms $\mathcal{W}_{k}^{\perp}F.$ It follows from the definition
(\ref{E:def_Wk}) that $\mathcal{W}_{k}^{\perp}F=\mathcal{D}^{\perp}H_{ (k)}$
where field $H_{ (k)}(x)$ is defined as $(\omega_{k}\cdot x)F(x)$, { $k=1,2,...d-1$}. Due to the
fast decay of $F(x)$ (see (\ref{E:decay})), fields $H_{ (k)}(x)$ satisfy
conditions of Lemma~\ref{T:Radon_of_div}. Therefore
\begin{align*}
\frac{\partial}{\partial p}\mathcal{W}_{k}^{\perp}F  &  =\frac{\partial
}{\partial p}\mathcal{D}^{\perp}(H_{ (k)})=\mathcal{R(}\operatorname{div}%
H_{ (k)})=\mathcal{R(}\omega_{k}\cdot F^{\mathrm{p}}+\omega_{k}\cdot
F^{\mathrm{s}}+(\omega_{k}\cdot x)\Phi(x))\\
&  =\omega_{k}\cdot\mathfrak{R}F^{\mathrm{p}}+\omega_{k}\cdot\mathfrak{R}%
F^{\mathrm{s}}+\mathcal{R}\{(\omega_{k}\cdot x)\Phi(x)\}.
\end{align*}
Due to (\ref{E:orthogonality}) term $\omega_{k}\cdot\mathfrak{R}F^{\mathrm{p}%
}$ vanishes, and one obtains%
\[
\omega_{k}\cdot\mathfrak{R}F^{\mathrm{s}}=\frac{\partial}{\partial
p}\mathcal{W}_{k}^{\perp}F-\mathcal{R}\{(\omega_{k}\cdot x)\Phi(x)\},\quad
k=1,...,d-1.
\]
These equations combined with (\ref{E:solenoid_zero}) determine projections of
vector-valued function $\mathfrak{R}F^{\mathrm{s}}$ onto the vectors of the
orthonormal basis $\mathfrak{B,}$ leading to the following result:%
\[
\mathfrak{R}F^{\mathrm{s}}=\sum_{k=1}^{d-1}\omega_{k}\left[  \frac{\partial
}{\partial p}\mathcal{W}_{k}^{\perp}F-\mathcal{R}\{(\omega_{k}\cdot
x)\Phi(x)\}\right]  .
\]
By combining the latter formula with equation (\ref{E:pot_part}) we arrive at
the formula (\ref{E:T_trans_2}) that gives an explicit expression for
$[\mathfrak{R}F](\omega,p).$ Since field components $F_{j}(x)$ are functions
from the Schwartz space, formula (\ref{E:FBP}) can be used to reconstruct
$F_{j}$'s from components of the vector-valued $[\mathfrak{R}F](\omega,p)$,
thus yielding equation (\ref{E:T_trans_3}). The proof of Theorem \ref{T:trans}
is complete.

\section{Properties of longitudinal transforms and proof of \label{S:long}%
Theorem \ref{T:long}}

In this section we assume that only longitudinal transforms $\mathcal{D}%
_{j}^{\shortparallel}F,\quad j=1,...,d-1,$ and one of the weighted
longitudinal transforms (e.g., $\mathcal{W}_{1}^{\shortparallel}F$) are known. Our goal is
to reconstruct field $F$ from these data.

\subsection{Reconstructing the solenoidal part of the field}

\begin{proposition}
Suppose $F$ is a smooth vector field satisfying the decay conditions
(\ref{E:decay}), and $F^{\mathrm{p}}$, $F^{\mathrm{s}}$ are its potential and
solenoidal parts, respectively. Then longitudinal transforms $\,\mathcal{D}%
_{j}^{\shortparallel}F^{\mathrm{p}}$ of $F^{\mathrm{p}}$ vanish, ${j}=1,...,d-1,$
and the Radon transform of the solenoidal part can be expressed through
$\mathcal{D}_{j}^{\shortparallel}F$ as follows:%
\begin{equation}
\lbrack\mathfrak{R}F^{\mathrm{s}}](\omega,p)=\sum\limits_{j=1}^{d-1}\omega
_{j}\,[\mathcal{D}_{j}^{\shortparallel}F](\omega,p),\qquad(\omega
,p)\in\mathbb{S}^{d-1}\times\mathbb{R}. \label{E:Radon_of_solenoidal}%
\end{equation}

\begin{proof}
Using (\ref{E:vectorRadon}) and (\ref{E:pot_part}) one obtains%
\[
\,\mathcal{D}_{j}^{\shortparallel}F\mathcal{=}\omega_{j}\cdot\mathfrak{R}%
F=\omega_{j}\cdot\mathfrak{R(}F^{\mathrm{p}}+F^{\mathrm{s}})=\omega_{j}%
\cdot\omega\mathcal{D}^{\perp}F^{\mathrm{p}}+\omega_{j}\cdot\mathfrak{R}%
F^{\mathrm{s}}=\omega_{j}\cdot\mathfrak{R}F^{\mathrm{s}}=\,\mathcal{D}%
_{j}^{\shortparallel}F^{\mathrm{s}},
\]
which implies that all longitudinal transforms of the potential part of a
field vanish:%
\[
\,\mathcal{D}_{j}^{\shortparallel}F^{\mathrm{p}}\mathcal{=}0,\qquad
j=1,...,d-1.
\]
Further, by expanding vector $\mathfrak{R}F$ in basis $\mathfrak{B}$ and using
(\ref{E:vectorRadon}) again see that%
\[
\mathfrak{R}F=\left(  \omega\cdot\mathfrak{R}F\right)  \omega+\sum
\limits_{j=1}^{d-1}\left(  \omega_{j}\cdot\mathfrak{R}F\right)  \omega
_{j}=\omega\mathcal{D}^{\perp}F+\sum\limits_{j=1}^{d-1}\omega_{j}%
\,\mathcal{D}_{j}^{\shortparallel}F.
\]
with
\[
\omega\mathcal{D}^{\perp}F=\omega\mathcal{D}^{\perp}F^{\mathrm{P}%
}=\mathfrak{R}F^{\mathrm{p}}\text{\qquad and\qquad}\omega_{j}\,\mathcal{D}%
_{j}^{\shortparallel}F=\omega_{j}\,\mathcal{D}_{j}^{\shortparallel
}F^{\mathrm{s}},\quad j=1,...,d-1,
\]
so that (\ref{E:Radon_of_solenoidal}) holds.
\end{proof}
\end{proposition}

Equation (\ref{E:Radon_of_solenoidal}) shows that the longitudinal transforms
$\,\mathcal{D}_{j}^{\shortparallel}F,$ $j=1,2,...,d-1$, contain enough
information to obtain the componentwise Radon transform of the solenoidal part
of the field. However, a straightforward componentwise application of the
inversion formula (\ref{E:FBP})\ is not justified in general, since components
of the field $F^{\mathrm{p}}$ are not in the Schwartz space. It is known that
formula~(\ref{E:FBP}) remains valid for slower decaying functions (see Chapter
1 of \cite{helgason2013Radon}). However, reconstruction of functions decaying
at the rate (\ref{E:thm_fields}) still, in general, cannot be guaranteed.
While we conjecture that inversion formula (\ref{E:FBP}) can be used for
componentwise inversion of \ (\ref{E:Radon_of_solenoidal}), we will not prove
this statement here. Instead, we notice that by computing the second derivative of
equation (\ref{E:Radon_of_solenoidal}) in $p$ one obtains the Radon transform
of the componentwise Laplacian $\Psi$ of $F^{\mathrm{s}}$:%
\begin{equation}
\mathfrak{R}\left(  \Psi\right)  =\frac{\partial^{2}}{\partial p^{2}%
}\mathfrak{R}F^{\mathrm{s}}=\frac{\partial^{2}}{\partial p^{2}}\sum
\limits_{j=1}^{d-1}\omega_{j}\,\mathcal{D}_{j}^{\shortparallel}F.
\label{E:Rad_of_Lapl_sol}%
\end{equation}
Let us find out the rate of decay of components of $\Psi$ at infinity.
Using (\ref{E:def_sol_part}) one obtains%
\begin{align*}
\Psi_{k}(x)  &  =\Delta F_{k}^{\mathrm{s}}(x)=\Delta F_{k}(x)-\Delta
F_{k}^{\mathrm{p}}(x)=\Delta F_{k}(x)-\Delta\frac{\partial}{\partial x_{k}%
}\varphi(x)=\Delta F_{k}(x)-\frac{\partial}{\partial x_{k}}\Phi(x)\\
&  =\Delta F_{k}(x)-\frac{\partial}{\partial x_{k}}\operatorname{div}%
F(x),\quad k=1,...,d.
\end{align*}
Since each component of field $F$ belongs to the Schwartz space, so does
$\Psi_{k}(x),$ $k=1,...,d.$ Therefore, equation (\ref{E:Rad_of_Lapl_sol}) can
be inverted componentwise using formula (\ref{E:FBP}), thus proving formula
(\ref{E:invert_for_sol}).

Knowing $\Psi$, the solenoidal part $F^{\mathrm{s}}$ of the field can be recovered as the
following convolution:%
\[
F^{\mathrm{s}}=G\ast\Psi.
\]

\subsection{Reconstructing the whole field}

In this section we will show that, assuming that $\Psi$ is known (for example,
reconstructed using formula (\ref{E:invert_for_sol})), the divergence $\Phi$
of the field can be reconstructed from the weighted longitudinal transform
$\mathcal{W}_{1}^{\shortparallel}\left(  F\right)  $ using formula
{ (\ref{E:invert_for_pot})}, and the whole field $F$ can be obtained as
convolutions (\ref{E:convolve_sol}).

As before, we will try to differentiate the weighted transform $\mathcal{W}%
_{1}^{\shortparallel}F.$ More precisely, let us evaluate the following
expression:%
\begin{align}
(e_{k}\cdot\omega)\frac{\partial}{\partial p}\,\mathcal{W}_{1}^{\shortparallel
}F  &  =(e_{k}\cdot\omega)\frac{\partial}{\partial p}\mathcal{R}\left(
(x\cdot\omega_{1})(\omega_{1}\cdot F(x))\right) \nonumber\\
&  =\mathcal{R}\left(  \frac{\partial}{\partial x_{k}}[(x\cdot\omega
_{1})(\omega_{1}\cdot F^{\mathrm{s}}(x))]\right)  +\mathcal{R}\left(
\frac{\partial}{\partial x_{k}}[(x\cdot\omega_{1})(\omega_{1}\cdot
F^{\mathrm{p}}(x))]\right)  \label{E:usingW_1}%
\end{align}
The second term in the right hand side of (\ref{E:usingW_1}) can be
transformed as follows:%
\begin{align}
\mathcal{R}\left(  \frac{\partial}{\partial x_{k}}[(x\cdot\omega_{1}%
)(\omega_{1}\cdot F^{\mathrm{p}}(x))]\right)   &  =\mathcal{R}\left(
\frac{\partial}{\partial x_{k}}[(x\cdot\omega_{1})(\omega_{1}\cdot
\nabla\varphi(x))]\right) \nonumber\\
&  =\mathcal{R}\left(  (e_{k}\cdot\omega_{1})(\omega_{1}\cdot\nabla
\varphi(x))\right)  +\mathcal{R}\left(  (x\cdot\omega_{1})\frac{\partial
}{\partial\omega_{1}}\frac{\partial\varphi(x)}{\partial x_{k}}\right)  .
\label{E:usingW_2}%
\end{align}
The first term in the right hand side of (\ref{E:usingW_2}) can be seen to be
equal to $(e_{k}\cdot\omega_{1})\,\mathcal{D}_{1}^{\shortparallel
}(F^{\mathrm{p}})$; it vanishes as a longitudinal transform of a potential
field. The remaining second term in (\ref{E:usingW_2}) can be simplified
further:%
\begin{align}
\mathcal{R}\left(  (x\cdot\omega_{1})\frac{\partial}{\partial\omega_{1}%
}\left(  \frac{\partial\varphi(x)}{\partial x_{k}}\right)  \right)   &
=\int\limits_{\mathbb{R}}...\int\limits_{\mathbb{R}}\left[
\int\limits_{\mathbb{R}}y_{1}\frac{\partial}{\partial y_{1}}\,\left(
\frac{\partial\varphi}{\partial x_{k}}(p\omega+y_{1}\omega_{1}+...+y_{d-1}%
\omega_{d-1})\,\right)  dy_{1}\right]  \,dy_{2}...\,dy_{d-1}\nonumber\\
&  =-\int\limits_{\mathbb{R}}...\int\limits_{\mathbb{R}}\left[
\int\limits_{\mathbb{R}}\frac{\partial}{\partial x_{k}}\varphi
(p\omega+y_{1}\omega_{1}+...+y_{d-1}\omega_{d-1})\,dy_{1}\right]
\,dy_{2}...\,dy_{d-1}\nonumber\\
&  =-\mathcal{R}\left(  \frac{\partial}{\partial x_{k}}\varphi(x)\right)
=-\mathcal{R}\left(  F_{k}^{\mathrm{p}}\right)  , \label{E:usingW_3}%
\end{align}
where integration by parts was performed with respect to $y_{1}$. By combining
(\ref{E:usingW_1}), (\ref{E:usingW_2}), and\ (\ref{E:usingW_3})\ we thus
obtain%
\begin{equation}
(e_{k}\cdot\omega)\frac{\partial}{\partial p}\,\mathcal{W}_{1}^{\shortparallel
}\left(  F\right)  =\mathcal{R}\left(  \frac{\partial}{\partial x_{k}}%
[(x\cdot\omega_{1})(\omega_{1}\cdot F^{\mathrm{s}}(x))]\right)  -\mathcal{R}%
\left(  F_{k}^{\mathrm{p}}\right)  . \label{E:usingW_4}%
\end{equation}
Now, let us apply the operator $(e_{k}\cdot\omega)\frac{\partial}{\partial p}$
again, this time to equation (\ref{E:usingW_4}):
\[
(e_{k}\cdot\omega)^{2}\frac{\partial^{2}}{\partial p^{2}}\,\mathcal{W}%
_{1}^{\shortparallel}\left(  F\right)  =\mathcal{R}\left(  \frac{\partial^{2}%
}{\partial x_{k}^{2}}[(x\cdot\omega_{1})(\omega_{1}\cdot F^{\mathrm{s}%
}(x))]\right)  -\mathcal{R}\left(  \frac{\partial}{\partial x_{k}}%
F_{k}^{\mathrm{p}}\right)  .
\]
By summing the above formula in $k$ from $1$ to $d$ one obtains
\begin{equation}
\frac{\partial^{2}}{\partial p^{2}}\,\mathcal{W}_{1}^{\shortparallel}\left(
F\right)  =\mathcal{R}\left(  \Delta\lbrack(x\cdot\omega_{1})(\omega_{1}\cdot
F^{\mathrm{s}}(x))]\right)  -\mathcal{R}\left(  \Phi\right)  .
\label{E:usingW_5}%
\end{equation}
Further, we note that
\begin{align*}
\Delta\lbrack(x\cdot\omega_{1})(\omega_{1}\cdot F^{\mathrm{s}}(x))]  &
=2\omega_{1}\cdot\nabla(\omega_{1}\cdot F^{\mathrm{s}}(x))+(x\cdot\omega
_{1})\Delta\lbrack\omega_{1}\cdot F^{\mathrm{s}}(x)]\\
&  =2\omega_{1}\cdot\nabla(\omega_{1}\cdot F^{\mathrm{s}}(x))+(x\cdot
\omega_{1})\omega_{1}\cdot\Psi(x).
\end{align*}
This allows one to simplify the first term in the right hand side of (\ref{E:usingW_5}) as
follows:%
\begin{align}
\mathcal{R}\left(  \Delta\lbrack(x\cdot\omega_{1})(\omega_{1}\cdot
F^{\mathrm{s}}(x)]\right)   &  =2\mathcal{R[}\omega_{1}\cdot\nabla(\omega
_{1}\cdot F^{\mathrm{s}}(x))]+\mathcal{R(}(x\cdot\omega_{1})(\omega_{1}%
\cdot\Psi(x)))\nonumber\\
&  =2\,\mathcal{D}_{1}^{\shortparallel}(\nabla(\omega_{1}\cdot F^{\mathrm{s}%
}(x))+\mathcal{R}\left(  (x\cdot\omega_{1})(\omega_{1}\cdot\Psi(x))\right)
=\mathcal{R}\left(  (x\cdot\omega_{1})(\omega_{1}\cdot\Psi(x))\right)  ,
\label{E:usingW_6}%
\end{align}
which holds since the longitudinal transform of a potential field
$\mathcal{D}_{1}^{\shortparallel}(\nabla(\omega_{1}\cdot F^{\mathrm{s}}(x))$
vanishes. By combining (\ref{E:usingW_5}) and (\ref{E:usingW_6}) we arrive at
the following formula%
\begin{equation}
\mathcal{R(}\Phi)=\mathcal{R}\left(  (x\cdot\omega_{1})(\omega_{1}\cdot
\Psi(x))\right)  -\frac{\partial^{2}}{\partial p^{2}}\,\mathcal{W}%
_{1}^{\shortparallel}(F) \label{E:another_nice_formula}%
\end{equation}
Now the Laplacian $\Phi$ of the potential $\varphi$ can be reconstructed by
inverting the Radon transform in~(\ref{E:another_nice_formula}), yielding
the whole field can be reconstructed by computing convolutions
(\ref{E:convolve_sol}). This completes the proof of theorem \ref{T:long}.

\section{Vector fields in magnetoacoustoelectric tomography\label{S:MAET}}

Our interest in the Radon transforms of vector field is motivated, in part, by
an inverse problem arising in magnetoacoustoelectric tomography (MAET). This
imaging modality is a novel coupled-physics technique designed to image the
electrical conductivity of biological objects. It is based on measurements of
electric potential arising in conductive tissues when they move in a magnetic field.
In detail, one places the object of interest in a strong constant magnetic
field and illuminates it with ultrasound pulses
\cite{Wen,Grasland,Roth,Zengin,Montalibet}. Frequently this is done with the
object immersed in conductive saline, which provides good acoustic coupling
and facilitates the measurements of the arising electric potential with the use of electrodes
immersed in the liquid. The said potential results from the interaction of the vibrational motion of
electrons and ions contained in a conductive tissue, with magnetic field.
This generates the Lorentz forces that separate the particles of opposite
polarities and, in turn, results in Ohmic current flowing through the
object and the saline. The electric potential associated with this current
is then measured outside of the object, providing the data for the future MAET
reconstruction.

\subsection{A traditional data acquisition scheme}
{ In the remaining part of the Section \ref{S:MAET} and in Section \ref{S:NUMERICS}
we work with the three-dimensional space.}

It has been shown (\cite{Montalibet}) that when the tissue with conductivity
$\sigma(x)$ moves with velocity $\mathbf{V}(t,x)$ within magnetic field
$\mathbf{B}$, the arising Lorentz force will generate Lorentz currents
{$\mathbf{J}^{L}(t,x)$} given by the formula%
\begin{equation}
\mathbf{J}^{L}(t,x)=\sigma(x)\mathbf{B}\times\mathbf{V}(t,x).
\label{E:new-Lorentz}%
\end{equation}
The vibrational velocity $\mathbf{V}(t,x)$ of the tissues arising due to the
ultrasound excitation is governed by the standard wave equation with the speed
of sound that can be assumed constant within soft tissues.
Without loss of generality the speed of sound can be set to 1. Then $\mathbf{V}%
(t,x)$ and the acoustic pressure $p(t,x)$ can be related to the velocity
potential $\zeta(t,x)$ by equations%
\[
\mathbf{V}(t,x)=\frac{1}{\rho}\nabla\zeta(t,x),\qquad p(t,x)=\frac{\partial
}{\partial t}\zeta(t,x).
\]
Here the density $\rho$ is assumed to be constant within soft tissues and
equal to the density of water. The scalar velocity potential $\zeta(t,x)$
itself also satisfies the wave equation in the whole space $\mathbb{R}^{3}$:%
\[
\Delta\zeta(t,x)=\frac{\partial^{2}}{\partial t^{2}}\zeta(t,x).
\]

The time scales of this model are such that the electromagnetic effects are
much faster than the mechanic motion of the liquid \cite{Zengin}. Therefore,
the currents in the system can be considered stationary, corresponding to
velocity $\mathbf{V}(t,x)$ at the given time $t.$ Then, it can be
shown  that the difference of potentials $M(t)$ measured by a
pair of electrodes can be expressed as follows \cite{Kun-MAET}%
\begin{equation}
M(t)=\frac{1}{\rho}\int\limits_{\Omega}\zeta(t,x)\mathbf{B}\cdot
\mathbf{C}(x)dx,\qquad\mathbf{C}(x)\equiv\nabla\times\mathbf{I}(x),
\label{E:new-meas}%
\end{equation}
where the lead current $\mathbf{I}(x)$ is the current that would flow through
the object in the absence of the magnetic and acoustic excitation, if a unit
potential difference were applied to the electrode pair. This quantity appears
in (\ref{E:new-meas}) because $\mathbf{I}(x)$ also describes the sensitivity
of the measuring system to a dipole placed at the point $x.$ Finally, the
domain $\Omega$ in the above equation is the volume occupied by the saline and
by the object immersed in it. Below, it will be convenient for us to consider a
model where $\Omega$ is large and can modeled by the whole space
$\mathbb{R}^{3}$. A measurement corresponding to a given acoustic wave
$\zeta(t,x)$ is, according to (\ref{E:new-meas}), a function of one variable.
The goal of MAET is, by using a sufficiently rich set of excitations
$\zeta(t,x)$, to collect enough information for reconstruction of  the conductivity
$\sigma(x)$ of the tissues.

In the early mathematical work on MAET \cite{Kun-MAET,Ammari2015}
mathematicians would assume that the object and the electrodes remain fixed
and the transducer is moved around the object providing a large family of
excitations $\zeta(t,x).$ Then the inverse problem of MAET naturally decouples
into two steps. Since curl $\mathbf{C}(x)$ is independent from $\zeta(t,x)$,
one considers (\ref{E:new-meas}) as values of projections of the quantity
$\mathbf{B}\cdot\mathbf{C}(x)$ on the complete set of excitations $\zeta
(t,x)$, and reconstructs $\mathbf{B}\cdot\mathbf{C}(x).$ Then, the second step
is to reconstruct the conductivity $\sigma(x)$ from $\mathbf{B}\cdot
\mathbf{C}(x)$, possibly from measurements repeated with two or three
different orientations of $\mathbf{B}$. Depending on the waveforms
$\zeta(t,x)$, the first step frequently can be reduced to one of the known
tomography problems. For example, if one illuminates the object by ideal
plane waves
\[
\zeta(t,x)=\delta(t-x\cdot\omega)
\]
with various directions $\omega\in\mathbb{S}^{2}$, the resulting measurements
can be expressed the Radon transform of $\mathbf{B}\cdot\mathbf{C}(x)$, that
can be easily inverted. Similarly, if one assumes an ideal point-like
transducer that produces spherical outgoing waves, the problem reduces to the
inverse source problem of thermo- and photoacoustic tomography, whose solution
is well known by now (see, e.g. \cite{Kuku,KuKuHand}).

The measuring scheme described above is easy to analyze. However, it does not
work well in practice. Indeed, if electrodes and the object are held in a
fixed position, there are very few directions from which the transducer can
send sound waves into the object without illuminating the electrodes, which generates
strong spurious elecrtic pulses that overwhelm the usefull signal. Thus, researchers are investigating a
different approach to data acquisition \cite{KIW-eng,Sun-rapid}, which assumes that
object is rotated while the electrodes are kept stationary. Equivalently, one
can keep the object fixed, and rotate the electrodes and transducer(s). In
both cases, the curl $\mathbf{C}(x)$ becomes a function of the object (or
electrodes') position, and the traditional two step reconstruction procedure described
above is not applicable anymore.

\begin{figure}[t]
\begin{center}
\includegraphics[scale=0.70]{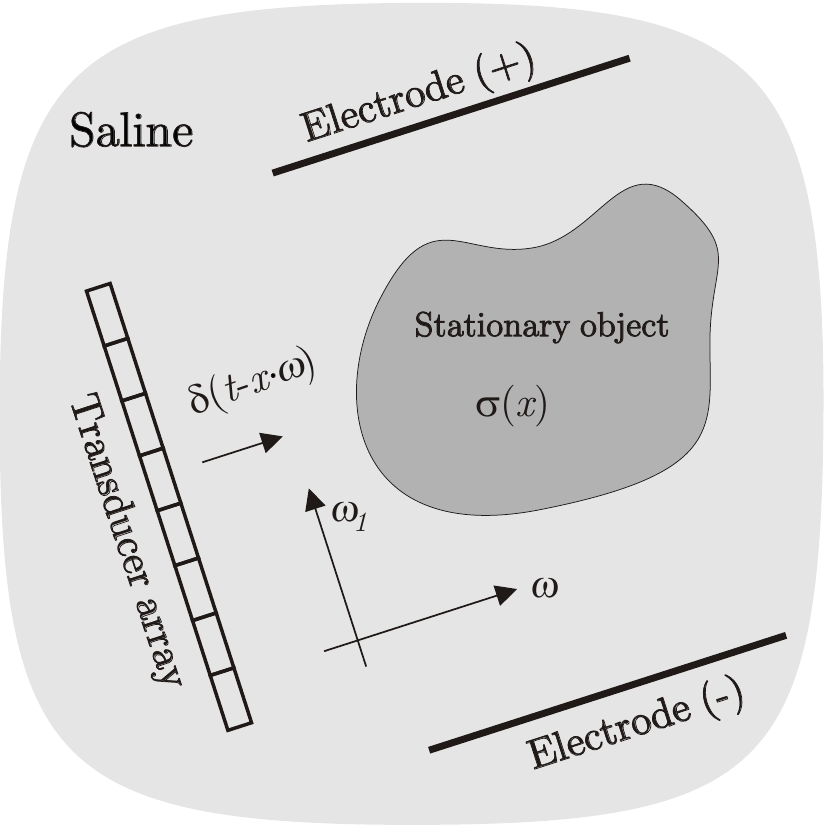}
\end{center}
\par
\vspace{-4mm}\caption{The novel MAET scheme; the electrode/transducer assembly
rotates around the object}%
\label{F:rot-geom}%
\end{figure}

\subsection{Acquisition scheme with a rotated object}

We thus consider here the novel acquisition scheme for MAET, with a rotating
electrode/transducer assembly, as shown in Figure~\ref{F:rot-geom}. The object
under investigation is immersed in a conductive saline, and the assembly
rotates around it. For simplicity, we model the propagation of currents in
this scheme assuming that the electrodes are placed far away from the object.
Here, the conductive medium is presumed to occupy all of $\mathbb{R}^{3},$
with the conductivity $\sigma(x)$ being constant and known outside of the
support $\Omega_{0}$ of the inhomogeneity, i.e. $\sigma(x)=\sigma_{0}$ for
$x\in\mathbb{R}^{3}\backslash\Omega_{0}.$ Then the lead current $\mathbf{I}$
is a function of $x$ and the orientation $\nu$ of the electrodes, i.e.
$\mathbf{I}\equiv\mathbf{I}_{\nu}(x)$. We assume that, in the absence of the
inhomogeneity, the electrodes generate field $E_{\nu}^{0}=\nu.$ In the presence
of inhomogeneity, additional potential $u_{\nu}(x)$ will arise, so that the
current can be expressed as
\begin{equation}
\mathbf{I}_{\nu}(x)=\sigma(x)(\nu+\nabla u_{\nu}(x)), \label{E:conductivity}
\end{equation}
subject to the following condition at infinity:%
\begin{equation}
\mathbf{I}_{\nu}(x)=\sigma_{0}E_{\nu}^{0}+o(1)=\sigma_{0}\nu+o(1)\text{ as
}x\rightarrow\infty. \notag %
\end{equation}
Due to the absence of sinks and sources of charges in the medium, current
$\mathbf{I}_{\nu}(x)$ is solenoidal. By setting to zero the divergence of
(\ref{E:conductivity}) we find that potential $u_{\nu}(x)$ solves the
divergence equation, subject to the decay at infinity%
\begin{align}
\nabla\cdot \left( \sigma(x){\nabla} u_{\nu}(x) \right) &  =-\nu\cdot\nabla\sigma(x),\quad
x\in\mathbb{R}^{3},\label{E:div-infty}\\
\lim_{x\rightarrow\infty}u_{\nu}(x)  &  =0. \label{E:MAETdecay}%
\end{align}

The above simplified model will allow us to express potential $u_{\nu}$ and
current $\mathbf{I}_{\nu}$ for an arbitrary orientation $\nu$ through three
"basis" solutions. Indeed, let us consider the solutions\ $u^{(j)}(x),$
$j=1,2,3$ of (\ref{E:div-infty}), (\ref{E:MAETdecay}) corresponding to
directions $\nu=e_{1},e_{2},e_{3}$, where $e_{_{j}}$'s are the canonical
vectors in $\mathbb{R}^{3}$:
\[
u^{(j)}(x)\equiv u_{e_{j}}(x),\quad j=1,2,3.
\]
The corresponding total currents and their curls will be denoted by
$\mathbf{I}^{(j)}(x)$ and $\mathbf{C}^{(j)}(x)$ respectively:%
\[
\mathbf{I}^{(j)}(x)=\sigma(x)(e_{j}+\nabla u^{(j)}(x)),\quad
\mathbf{C}^{(j)}(x)=\nabla\times\mathbf{I}^{(j)}(x),\quad j=1,2,3.
\]
Due to the linearity of the problem (\ref{E:div-infty}), (\ref{E:MAETdecay})
with respect to the { right} hand side of (\ref{E:div-infty}), for an arbitrary
direction $\nu$ the potential $u_{\nu}(x)$ and current $\mathbf{I}_{\nu
}(x)$ can be represented as the following linear combinations:
\begin{align}
u_{\nu}(x)  &  =\sum_{j=1}^3 (e_{j}\cdot\nu)u^{(j)}(x),\nonumber\\
\mathbf{I}_{\nu}(x)  &  =\sum_{j=1}^3 (e_{j}\cdot\nu)\mathbf{I}^{(j)}%
(x)=\sigma(x)\sum_{j=1}^3 (e_{j}\cdot\nu)(\nu+\nabla u^{(j)}(x))=\sigma
(x)(\nu+\sum_{j=1}^3 (e_{j}\cdot\nu)\nabla u^{(j)}(x)). \label{E:current_combi}%
\end{align}

Let us denote by $\mathbf{C}_{\nu}(x)$ the curl of the three-dimensional field
$\mathbf{I}_{\nu}(x)$:%
\[
\mathbf{C}_{\nu}(x)=\nabla\times\mathbf{I}_{\nu}(x).
\]
Recall that MAET measurements are directly related to $\mathbf{B\cdot C}_{\nu
}(x)$ (see equation (\ref{E:new-meas})). Let us assume for now that the
transducer is oriented along the vector $\omega$ perpendicular to $\nu$, and
is producing ideal plane waves. Then, the corresponding measurements
$M(\nu,\omega,t)$ can be expressed as
\begin{equation}
M(\nu,\omega,t)=\frac{1}{\rho}\int\limits_{\Omega_{0}}\delta(t-x\cdot
\omega)\mathbf{B}\cdot\mathbf{C}_{\nu}(x)dx. \label{E:new_measuremnt}%
\end{equation}
Here the integration is restricted to $\Omega_{0}$ since the curls
$\mathbf{C}^{(j)}(x)$ of currents $\mathbf{I}^{(j)}(x)$ vanish within any
region with constant conductivity, i.e. outside of $\Omega_{0}$. By combining
equations (\ref{E:current_combi}) and\ (\ref{E:new_measuremnt}) one obtains
\[
M(\nu,\omega,t)=\frac{1}{\rho}\int\limits_{\Omega_{0}}\delta(t-x\cdot
\omega)\mathbf{B}\cdot\sum_{j=1}^3(e_{j}\cdot\nu)\mathbf{C}^{(j)}(x)  dx=\int%
\limits_{\Omega_{0}}\delta(t-x\cdot\omega)\nu\cdot\mathfrak{C}(x)  { dx},
\]
where we introduced the vector field $\mathfrak{C}(x)$ defined as follows%
\[
\mathfrak{C}(x)=\frac{1}{\rho}\left(  \mathbf{B}\cdot\mathbf{C}^{(1)}%
,\mathbf{B}\cdot\mathbf{C}^{(2)},\mathbf{B}\cdot\mathbf{C}^{(3)}\right)  (x).
\]
We thus recognize $M(\nu,\omega,t)$ as a longitudinal Radon transform of the
vector field $\mathfrak{C}(x)$. If one directs vector $\nu$ to be parallel to one of
the vectors $\omega_{1}$ or $\omega_{2}$ orthogonal to $\omega,$ measurements
$M(\nu,\omega,t)$ coincide with the longitudinal transforms $\left[
\mathcal{D}_{k}^{\shortparallel}\mathfrak{C}\right]  (\omega,p)$ defined by
equation (\ref{E:def_longitudinal}):%
\[
M(\omega_{k},\omega,t)=\left[  \mathcal{D}_{k}^{\shortparallel}F\right]
(\omega,t),\quad k=1,2.
\]

If one manages to reconstruct from MAET measurements field $\mathfrak{C}(x),$
projections of curls $\mathbf{B}\cdot\mathbf{C}^{(j)}$ are easily found:%
\[
\mathbf{B}\cdot\mathbf{C}^{(j)}(x)=\rho e_{j}\cdot\mathfrak{C}(x),\quad
j=1,2,3.
\]
Then, { the} measurements can be repeated with alternatively directed
$\mathbf{B}$, until $\mathbf{C}^{(j)}$'s can be determined. After that,
currents $\mathbf{I}^{(j)}$ and conductivity $\sigma(x)$ can be reconstructed,
following the techniques presented in \cite{Kun-MAET,Ammari2015}. In a
simplified two-dimensional setting (as in \cite{KIW-eng}), curls
$\mathbf{C}^{(j)},$ $j=1,2,$ are oriented orthogonally to the plane in which
currents are flowing, and magnetic induction $\mathbf{B}$ is parallel to
$\mathbf{C}^{(j)}$'s. Additional directions of $\mathbf{B}$ are not needed in
this case.

However, analysis presented in the previous sections of this paper shows that
only a solenoidal part of a vector field $\mathfrak{C}(x)$ can be
reconstructed from known longitudinal transforms $\mathcal{D}_{k}%
^{\shortparallel}\mathfrak{C},\quad k=1,2.$ In general, there is no reason to
expect that field $\mathfrak{C}(x)$ is solenoidal. As a way to remedy this
situation, we propose to conduct additional measurements, by illuminating the
object with linearly modulated acoustic waves in the form
\begin{equation}
\zeta(t,x)=(x\cdot\omega_{1})\delta(t-x\cdot\omega), \label{E:new_waves}%
\end{equation}
with directions $\omega$ varying over $\mathbb{S}^{2}$, and $\omega_{1}$
aligned with the electrode directions. Such measurements $N(\omega_{1}%
,\omega,t)$ are described by the formula
\[
N(\omega_{1},\omega,t)=\frac{1}{\rho}\int\limits_{\Omega_{0}}(x\cdot\omega
_{1})\delta(t-x\cdot\omega)\mathbf{B}\cdot\sum_{j=1}^3 (e_{j}\cdot\nu
)\mathbf{C}^{(j)}(x)dx=\int\limits_{\Omega_{0}}(x\cdot\omega_{1}%
)\delta(t-x\cdot\omega)(\omega_{1}\cdot\mathfrak{C}(x))dx;
\]
they can be expressed as the weighted longitudinal transform $\mathcal{W}%
_{1}^{\shortparallel}\mathfrak{C}$:
\[
N(\omega_{1},\omega,t)=\left[  \mathcal{W}_{k}^{\shortparallel}\mathfrak{C}%
\right]  (\omega,t).
\]
Theorem \ref{T:long} states that the vector field $\mathfrak{C}(x)$ can be
reconstructed from its longitudinal transforms $\mathcal{D}_{1}%
^{\shortparallel}\mathfrak{C}$ and $\mathcal{D}_{2}^{\shortparallel
}\mathfrak{C}$ and weighted longitudinal transform $\mathcal{W}_{1}%
^{\shortparallel}\mathfrak{C}$ using formulas (\ref{E:invert_for_sol}%
)-(\ref{E:convolve_sol}).

MAET\ measurements using linearly modulated waves (\ref{E:new_waves}) have not
been implemented previously, in part because the benefit of such measurements
have not been previously discussed in the literature. However, there is no
physical obstacles for conducting such an experiment. Indeed, functions in the
form (\ref{E:new_waves}) are easily seen to satisfy the wave equation. They
can be generated in a number of ways. For example, if a transducer array is
used for sound generation (as in \cite{Sun-rapid}), such waves can be obtained
by scaling linearly the excitation voltage along the transducer elements. If a
synthetic flat detector is utilized (as in \cite{KIW-eng}), one obtains the
desired result by a weighted averaging of individual measurements. Such sound
waves can also be excited using optically generated ultrasound
\cite{Xia-MAET-lgus,Xia-MAET-lgus-SPIE}, by using optical excitation with
linearly varying intensity.

We will not attempt to simulate a full MAET experiment with linearly modulated
sound waves in this paper, leaving it to the future work. Below we present
numerical simulations of reconstruction of a 3D vector field from its
longitudinal transforms $\mathcal{D}_{1}^{\shortparallel}F$, $\mathcal{D}%
_{2}^{\shortparallel}F$ and the weighted longitudinal transform $\mathcal{W}%
_{1}^{\shortparallel}F$.
\begin{figure}[t]
\noindent \begin{tabular}{ccc}
\includegraphics[scale=0.48]{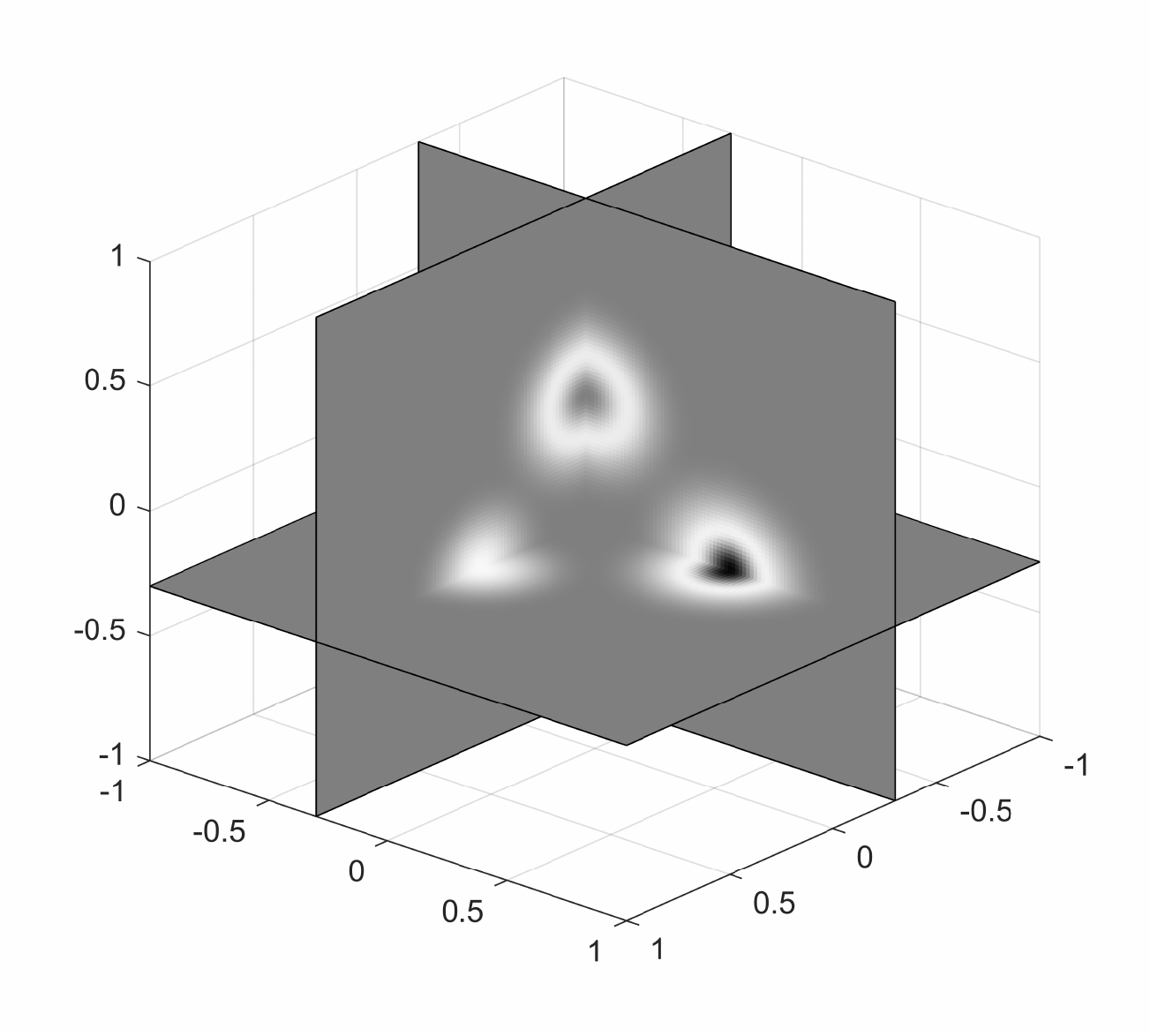} \hspace{-10mm}&
\hspace{-10mm} \includegraphics[scale=0.48]{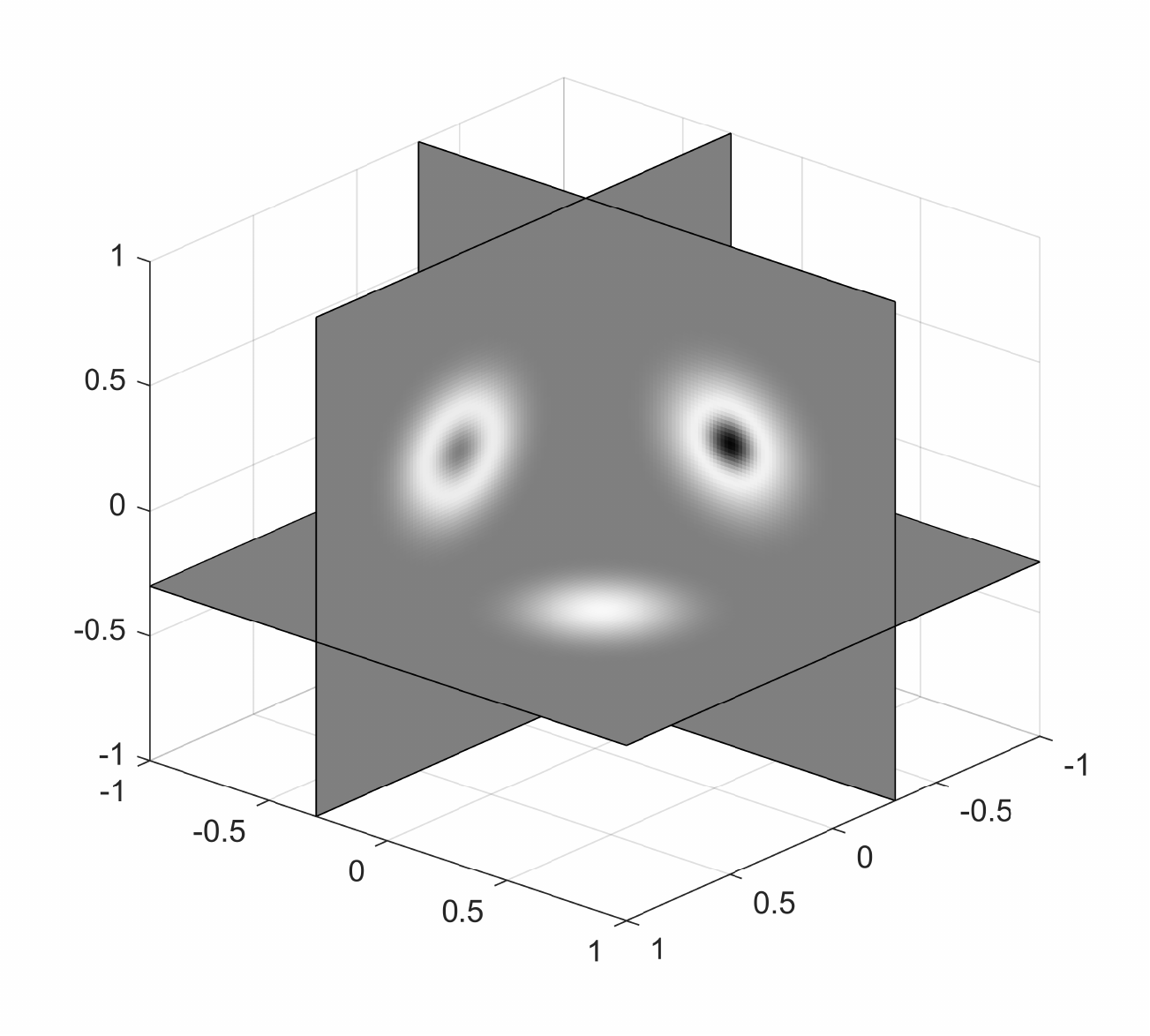} \hspace{-10mm}&
\hspace{-10mm} \includegraphics[scale=0.48]{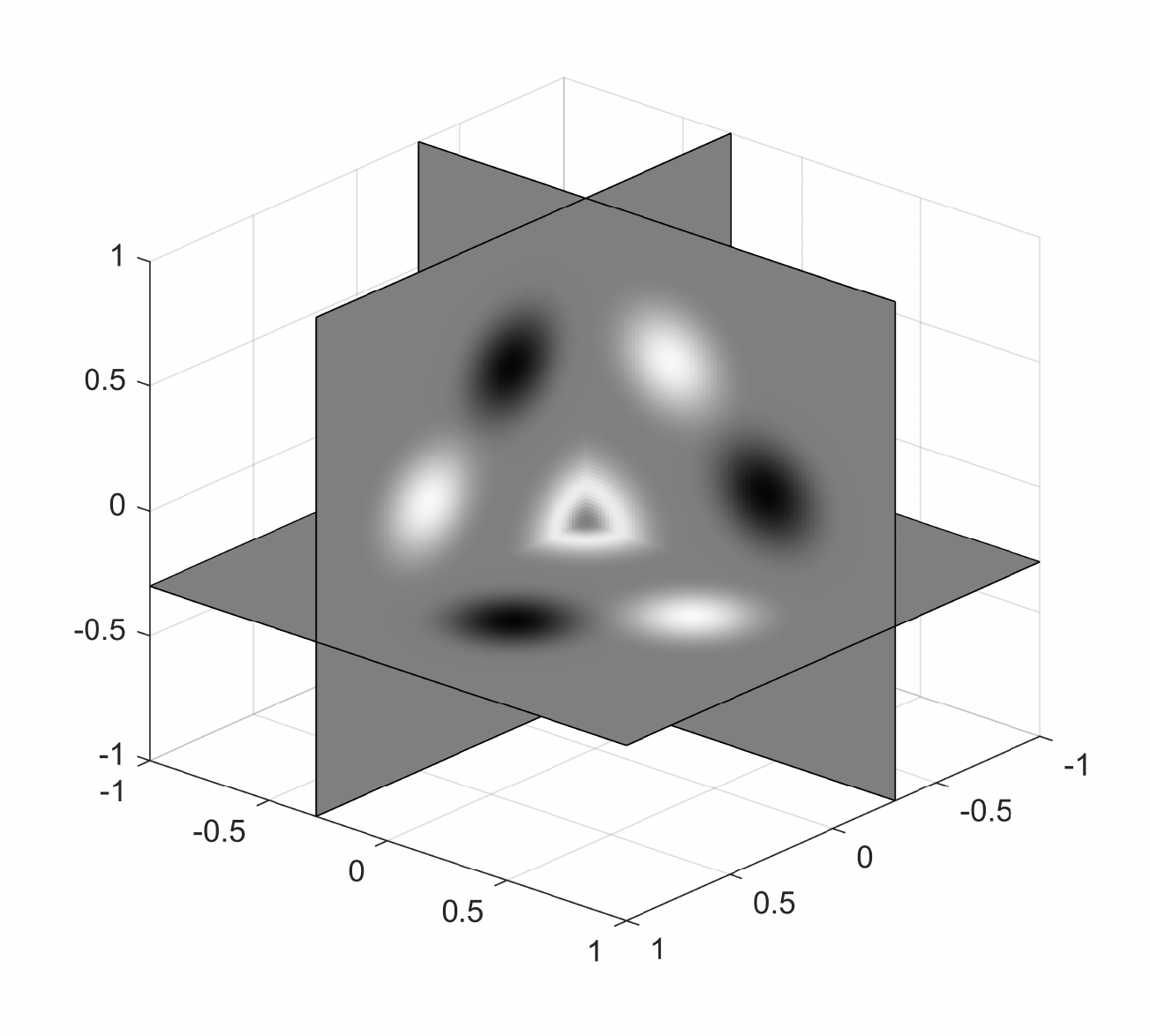}  \\
$F_1(x)$ &  $F_2(x)$  & $F_3(x)$ \\
\end{tabular}
\begin{center}
\includegraphics[scale=0.7]{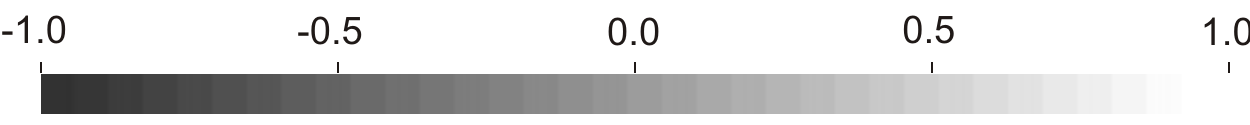}
\end{center}
\par
\caption{Components of vector field $F$ and the gray scale we use throughout the paper}
\label{F:phantom}%
\end{figure}

\section{Numerical simulations\label{S:NUMERICS}}

The goal of this section is to demonstrate the validity of the exact
reconstruction formulas (\ref{E:invert_for_sol})-(\ref{E:convolve_sol}) in a
numerical experiment. To this end we picked a smooth phantom $F(x)$
defined in the unit ball $B(1,0)$ in $\mathbb{R}^{3}.$ Each component
$F_{j}(x),$ $j=1,2,3$ is a linear combination of a rather arbitrary collection
of shifted radially symmetric functions ("bumps")%
\begin{align*}
F_{j}(x)  &  =\sum_{k=1}^{M_{j}}a_{k,j}f\left(  x-x_{k,j}^{(c)},R_{k,j}%
\right)  ,\\
f(r,R) &  =\left\{
\begin{array}
[c]{cc}%
\left(  1-\frac{r^{2}}{R}\right)  ^{4}, & r<R,\\
0, & r\geq R,
\end{array}
\right.
\end{align*}
where $a_{k,j}$'s are weights, and $x_{k,j}^{(c)}$ are the centers and
$R_{k,j}$ are the radii of support of the corresponding bumps. For the ease of
visualization, all centers $x_{k,j}^{(c)}$ were chosen to lie in one of the
planes $x_{1}=-0.3,$ $x_{2}=-0.3,$ or $x_{3}=-0.3.$ Each so defined component
$F_{j}$ is a $C^{3}(B(1,0))$ function. The phantom is shown in Figure~\ref{F:phantom},
{ and the values of constants $x_{k,j}$, $R_{k,j}$, and $a_{k,j}$ used in our simulations
can be found in Table~1. In addition, $M_1=5$, $M_2=5$, $M_3=8$.}
\begin{table}[t]
\begin{center}
\begin{tabular}{|r|r|c|r|r|}
\hline
$j$ & $k$ & $x_{k,j}$ & $R_{k,j}$ & $a_{k,j}$ \\ \hline
 $1$ & $1$ &$(0.2,-.3,-.3)         $  &  $0.4  $& $1.0   $ \\
 $1$ &2$ $ &$( -.3 ,  -.3  ,  0.2 )$  &  $  0.5$& $  1.7 $ \\
 $1$ &3$ $ &$ (-.3 ,  -.3  ,  0.2 )$  &  $ 0.25$& $ -1.7 $ \\
 $1$ &4$ $ &$ (-.3 ,  0.3  ,  -.3 )$  &  $  0.5$& $  1.5 $ \\
 $1$ &5$ $ &$ (-.3 ,  0.3  ,  -.3 )$  &  $  0.2$& $  -2.5$ \\   \hline
 $2$ &1$ $ &$(0.2  ,  0.2  ,  -.3 )$  &  $ 0.5 $& $ 1.0  $ \\
 $2$ &2$ $ &$ (-.3 ,  0.3  ,  0.2 )$  &  $ 0.5 $& $ 1.5  $ \\
 $2$ &3$ $ &$ (-.3 ,  0.3  ,  0.2 )$  &  $ 0.2 $& $ -2.5 $ \\
 $2$ &4$ $ &$ (0.3 ,  -.3  ,  0.2 )$  &  $ 0.5 $& $ 1.7  $ \\
 $2$ &5$ $ &$ (0.3 ,  -.3  ,  0.2 )$  &  $0.25 $& $-1.7  $ \\    \hline
 $3$ &1$ $ &$(-.3  ,  -.3  ,  -.3 )$  &  $0.45 $& $  1.5 $ \\
 $3$ &2$ $ &$(-.3  ,  -.3  ,  -.3 )$  &  $ 0.2 $& $ -1.5 $ \\
 $3$ &3$ $ &$(-.3  ,  .05  ,  .45 )$  &  $ 0.4 $& $  1.0 $ \\
 $3$ &4$ $ &$(-.3  ,  .45  ,  .05 )$  &  $ 0.4 $& $  -1.0$ \\
 $3$ &5$ $ &$(.05  ,  -.3  , .45  )$  &  $ 0.4 $& $  -1.0$ \\
 $3$ &6$ $ &$(.45  ,  -.3  , .05  )$  &  $ 0.4 $& $  1.0 $ \\
 $3$ &7$ $ &$(.05  ,  .45  ,  -.3 )$  &  $ 0.4 $& $  1.0 $ \\
 $3$ &8$ $ &$(.45  ,  .05  ,  -.3 )$  &  $ 0.4 $& $  -1.0$ \\     \hline
\end{tabular}%
\end{center}
\caption{Values of constants $x_{k,j}$, $R_{k,j}$, and $a_{k,j}$ used in both simulations} \vspace{-10mm}
\end{table}

\begin{figure}[t]
\begin{tabular}{ccc}
\includegraphics[scale=0.48]{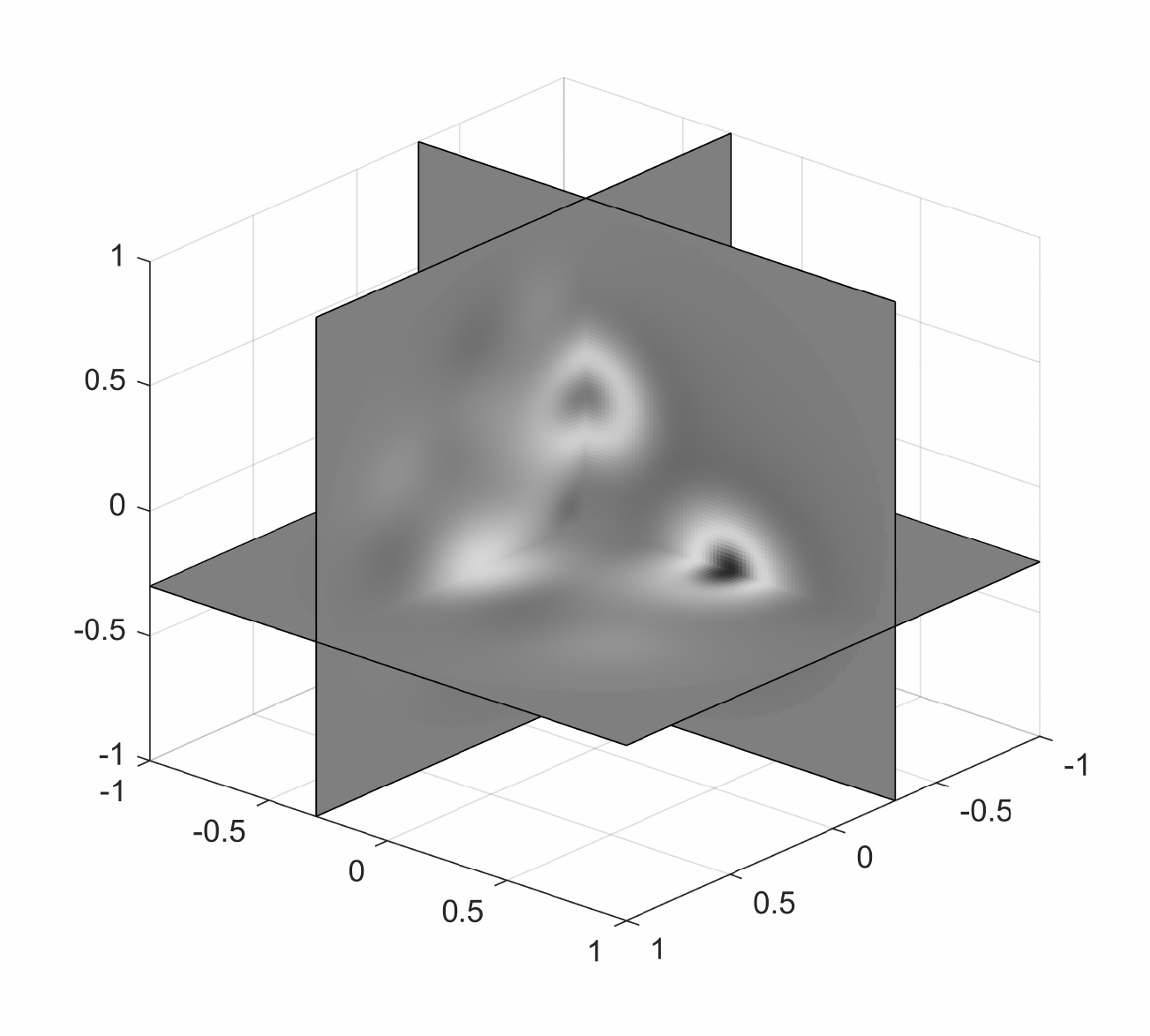} \hspace{-10mm}&
\hspace{-10mm} \includegraphics[scale=0.48]{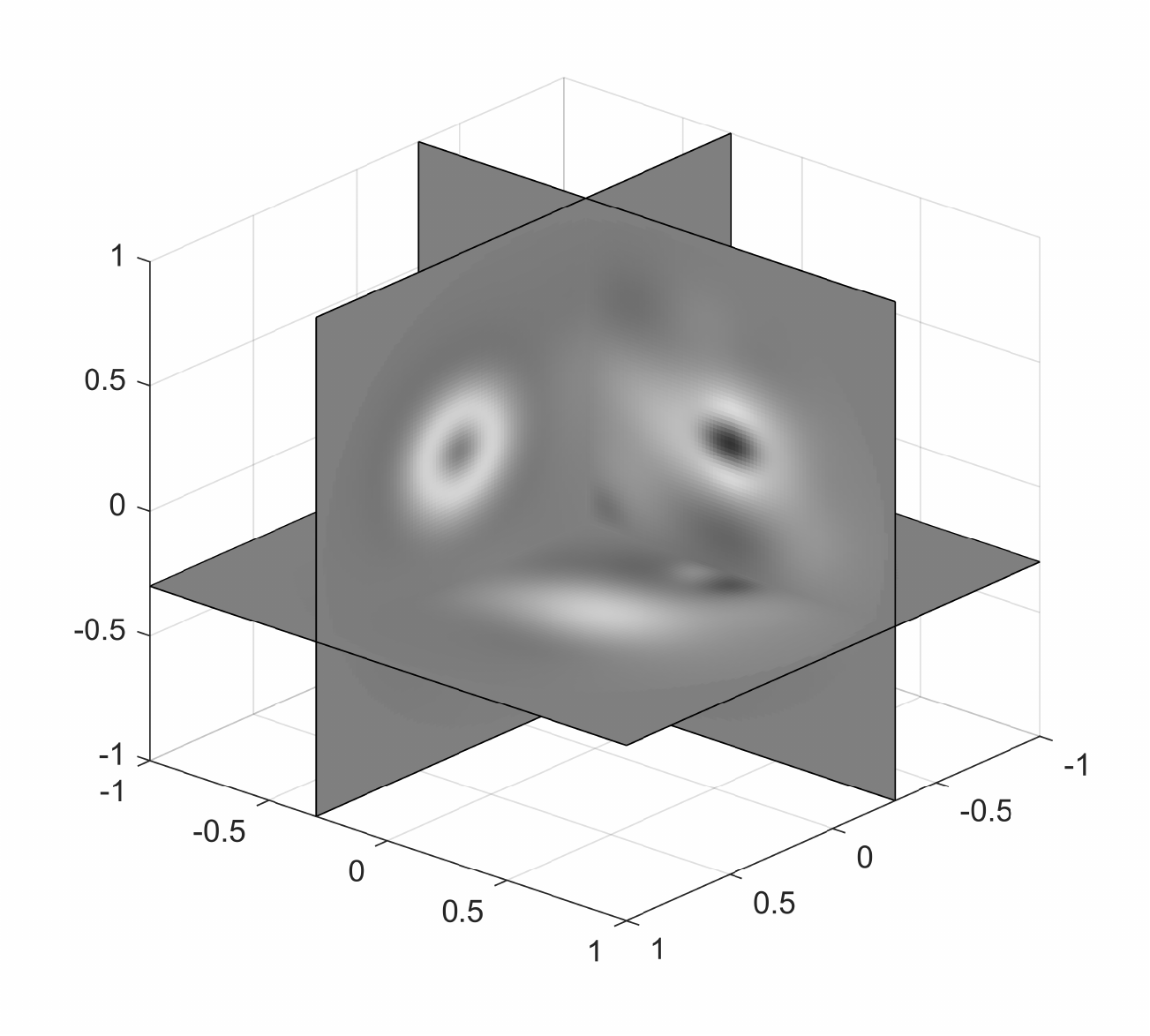} \hspace{-10mm}&
\hspace{-10mm} \includegraphics[scale=0.48]{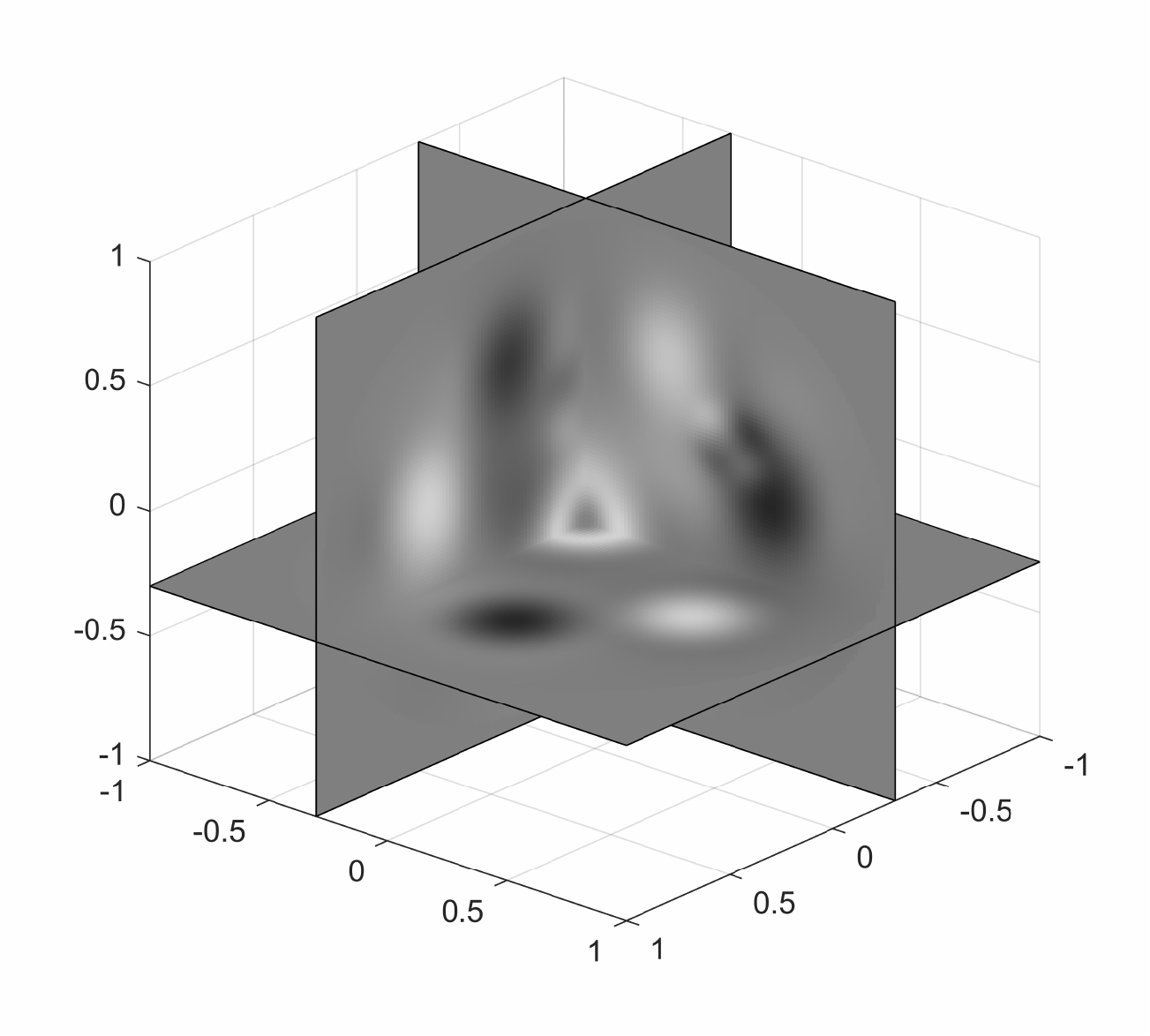}  \\
$F^\mathrm{s}_1(x)$ &  $F^\mathrm{s}_2(x)$  & $F^\mathrm{s}_3(x)$ \\
\\
\includegraphics[scale=0.48]{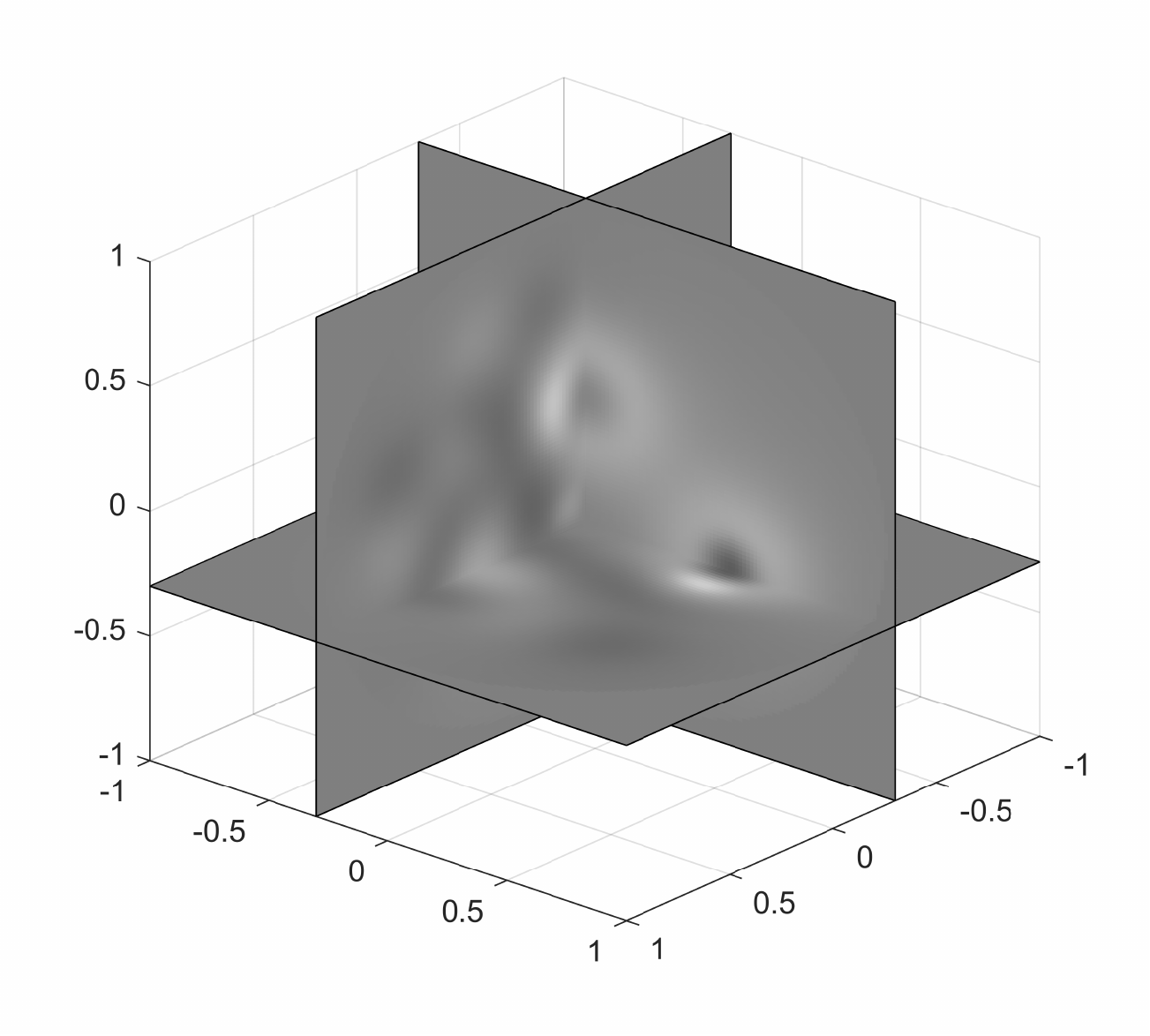} \hspace{-10mm}&
\hspace{-10mm} \includegraphics[scale=0.48]{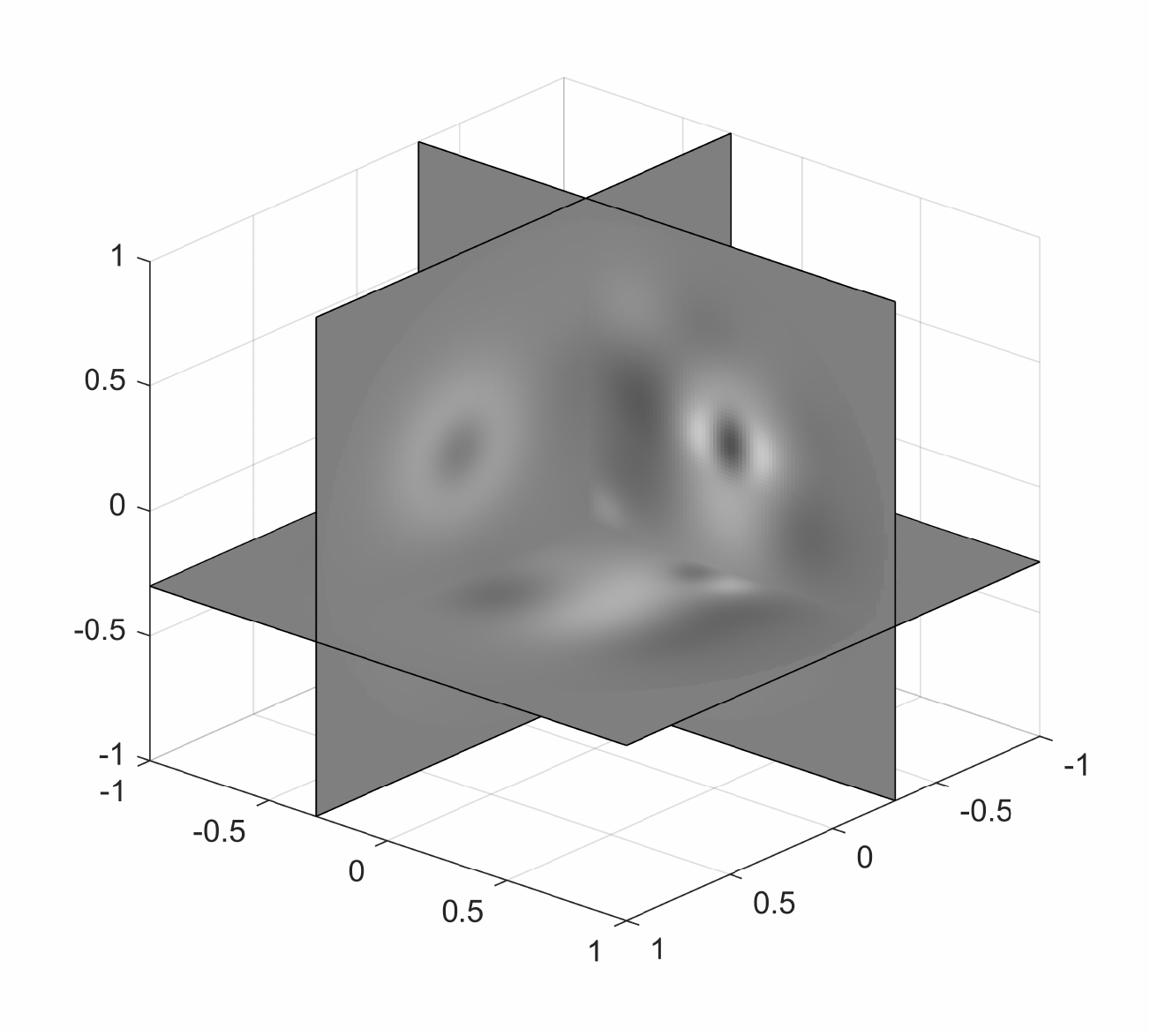} \hspace{-10mm}&
\hspace{-10mm} \includegraphics[scale=0.48]{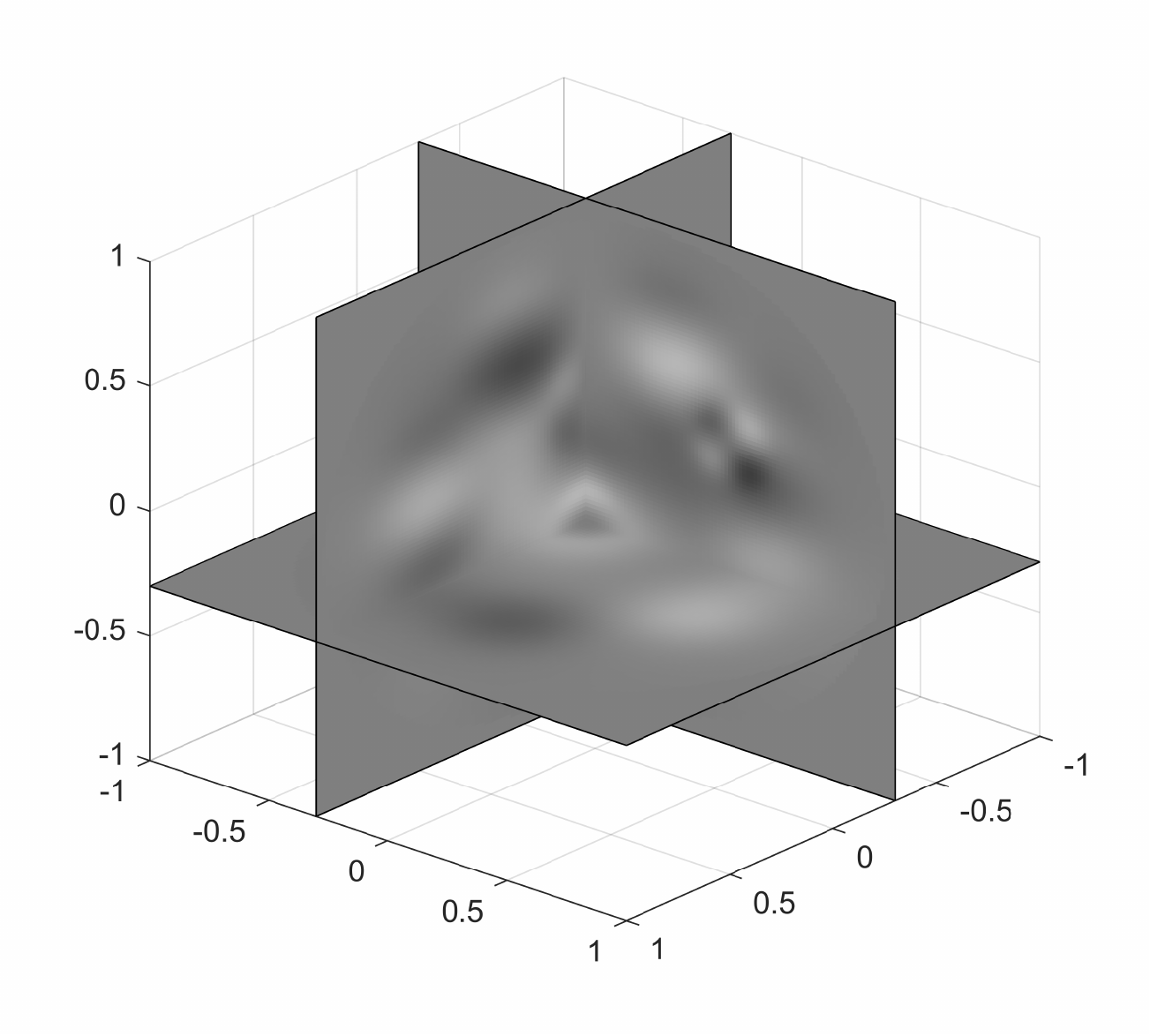}  \\
$F^\mathrm{p}_1(x)$ &  $F^\mathrm{p}_2(x)$  & $F^\mathrm{p}_3(x)$ \\
\\
\end{tabular}
\par
\caption{Reconstructed solenoidal and potential parts of the field,  $F^\mathrm{s}$ and  $F^\mathrm{p}$ }
\label{F:rec_nonoise}%
\end{figure}
Formulas (\ref{E:def_Wk})-(\ref{E:vectorRadon}) show that one can find values
of longitudinal transforms $\mathcal{D}_{1}^{\shortparallel}F$, $\mathcal{D}%
_{2}^{\shortparallel}F$, and $\mathcal{W}_{1}^{\shortparallel}F$ by computing
the standard and the linearly weighted Radon transforms of each of the $F_{j}%
$, $j=1,2,3.$ The latter transforms of the radial bump functions $f\left(
x-x_{k,j}^{(c)},R_{k,j}\right)  $ are given by the following formulas that can
be obtained by elementary calculations:%
\[
\left[  \mathcal{R}f\left(  x-x_{k,j}^{(c)},R_{k,j}\right)  \right]
(\omega,p)=\frac{\pi}{5}R_{k,j}^{2}\left(  1-\frac{\left(  p-\omega\cdot
x_{k,j}^{(c)}\right)  ^{2}}{R_{k,j}^{2}}\right)  ^{5},\text{ if }\left\vert
p-\omega\cdot x_{k,j}^{(c)}\right\vert <R_{k,j},\text{ }0\text{ otherwise,}%
\]
and
\[
\left[  \mathcal{R}\left\{  \mathcal{(}\omega_{1}\cdot x)f\left(
x-x_{k,j}^{(c)},R_{k,j}\right)  \right\}  \right]  (\omega,p)=\left(
\omega_{1}\cdot x_{k,j}^{(c)}\right)  \left[  \mathcal{R}f\left(
x-x_{k,j}^{(c)},R_{k,j}\right)  \right]  (\omega,p).
\]

While our formulas are valid for any choice of orthonormal basis vectors
$\omega_{1}(\omega)$ and $\omega_{2}(\omega)$, for numerical simulations we defined
these vectors as follows. Vector $\omega_{2}$ was chosen to lie in the
horizontal plane spanned by canonical vectors $e_{1}$ and $e_{2}$; it was
computed as follows:
\[
\omega_{2}(\omega)=\frac{\omega_{2}^{\ast}(\omega)}{|\omega_{2}^{\ast}%
(\omega)|},\text{ where }\omega_{2}^{\ast}(\omega)=(-(\omega\cdot
e_{2}),(\omega\cdot e_{1}),0).
\]
{ The directions of $\omega$ were discretized in such a way (see the next paragraph),
that the values  $(0,0,1)$ and $(0,0,-1)$ were never used, and the above formula for $\omega_2$
was always well defined.}
Vector $\omega_{1}(\omega)$ was computed as the cross-product $\omega
_{1}(\omega)=\omega\times\omega_{2}(\omega).$

The following grid in the variables $(\omega,p)$ was used to compute the Radon
transforms. Variable $p$ was discretized using a uniform gird with $257$ nodes
in the interval $[-1,1].$ Vector $\omega(\theta,\varphi)=(\sin\varphi
\cos\theta,\sin\varphi\sin\theta,\cos\varphi)$ was discretized using a product
grid on { $[0, 2 \pi] \times [0,\pi]$}, with $513$ uniformly spaced nodes in the variable $\theta$ and $256$
Gaussian nodes in the variable $t=\cos\varphi$.
{ For simplicity of presentation we did not use the redundancy in the Radon transform to reduce the
required data and the computational complexity. However, in practice it is sufficient to vary
$\omega$ over half a sphere and multiply the result by the factor of 2. }

The inversion of the classical Radon transform required by equations
(\ref{E:invert_for_sol}) and (\ref{E:invert_for_pot}) was implemented by
discretizing the 3D version of the formula (\ref{E:FBP}), with $\alpha=0$:%
\begin{equation}
f(x)=[\mathcal{R}^{-1}g](x)=\frac{1}{8\pi^{2}}\left[  \mathcal{R}^{\#}\left(
\frac{\partial}{\partial^{2}p}g(\omega,p)\right)  \right]  (x),\qquad x\in
B(1,0).\label{E:3FBP}%
\end{equation}
The inversion was computed in the nodes of $257\times257\times257$ Cartesian
grid in $x$, for $ { |x| }\le 1$ only.
For our first simulation, the derivatives in $p$ in
(\ref{E:invert_for_pot}) and in (\ref{E:3FBP}) were computed by a spectrally
accurate algorithm, using the Fast Fourier transform (FFT), in order to achieve
high accuracy when processing theoretically exact data. The components of the
reconstructed fields $F^{\mathrm{s}}(x)$ and $F^{\mathrm{p}}(x)$ are shown in
Figure~\ref{F:rec_nonoise} (the gray scale used in the images is the same as in Figure~\ref{F:phantom}).
When added together, these fields produce an accurate
approximation to the exact $F(x)$. When plotted in a grey scale figure (not shown here) the
reconstructed $F(x)$ is indistinguishable from the exact field presented in
Figure~\ref{F:phantom}. In this case, the relative $L_{2}$ error of the reconstruction is 0.09\% and the
relative $L_{\infty}$ error does not exceed 0.3\%. This is consistent with the
exactness of our reconstruction formulas.

Our second numerical simulation aims to demonstrate the noise
sensitivity of formulas (\ref{E:invert_for_sol})-(\ref{E:convolve_sol}). In
the above mentioned equations, functions $\Phi$ and $\Psi_{j},  { \quad j=1,2,...,d}$ are reconstructed from
the second derivatives  of the data in $p$. This is followed by convolutions
with smoothing kernels in (\ref{E:convolve_sol}). However, the solenoidal part
$F^{\mathrm{s}}$ of the field is obtained by convolutions with the
fundamental solution $(\Psi_{j}\ast G),\quad j=1,...,{d},$ whereas the
potential part is computed by convolution of $\Phi$ with the gradient $\nabla G$
of the fundamental solution. This additional differentiation implies that the
potential part should be more sensitive to high spatial frequencies of the
noise.

\begin{figure}[t]
\begin{tabular}{ccc}
\includegraphics[scale=0.48]{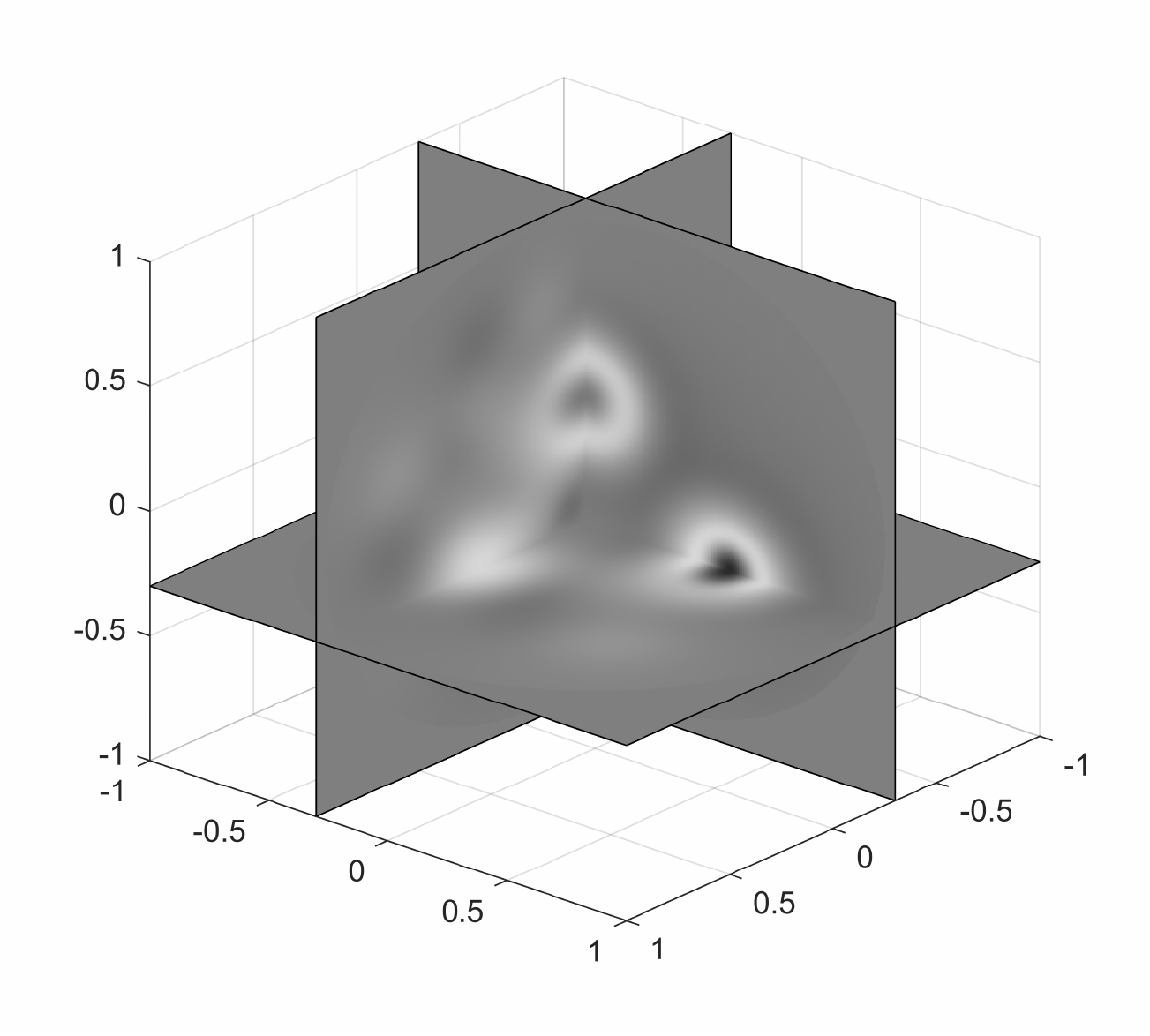} \hspace{-10mm}&
\hspace{-10mm} \includegraphics[scale=0.48]{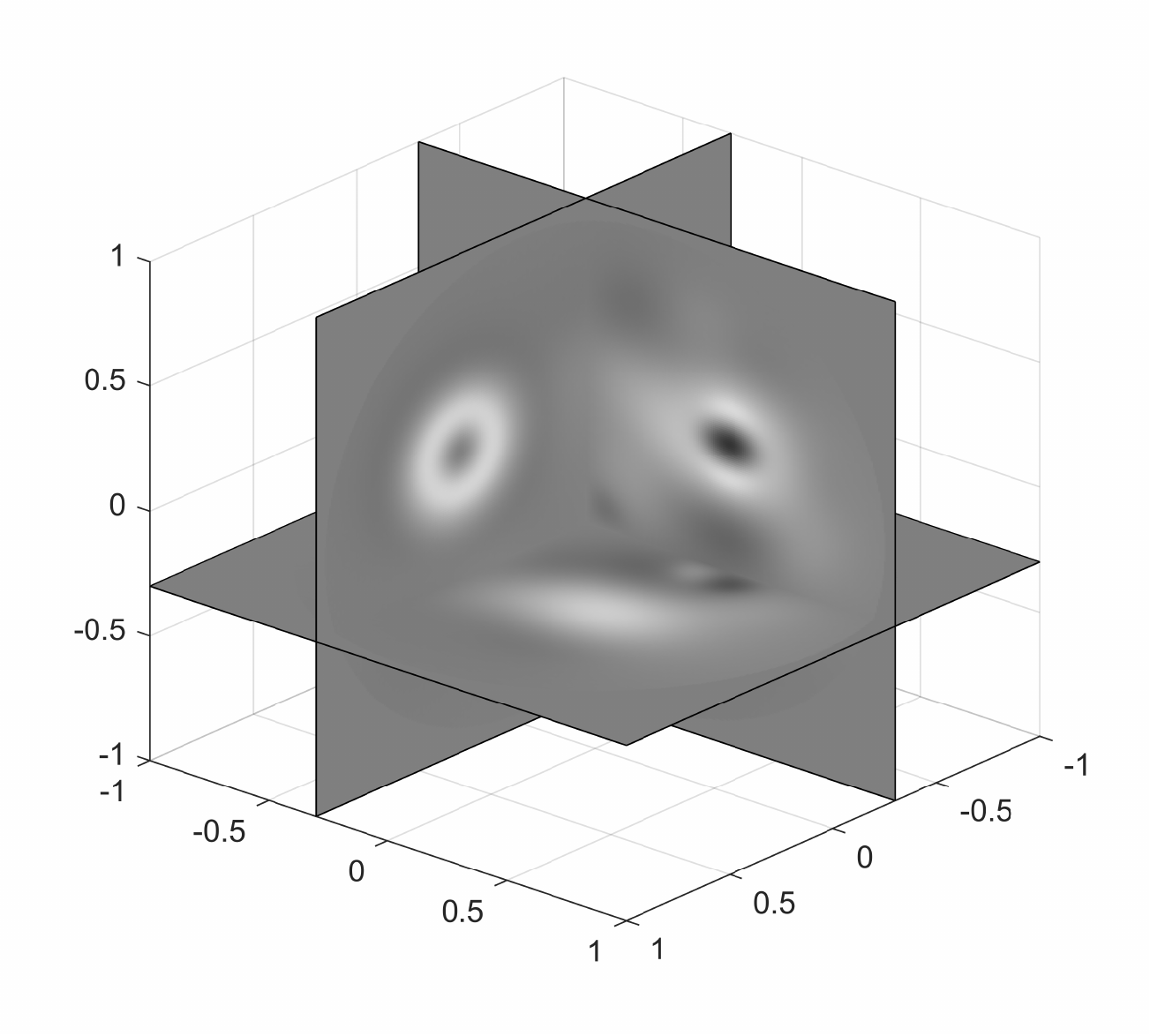} \hspace{-10mm}&
\hspace{-10mm} \includegraphics[scale=0.48]{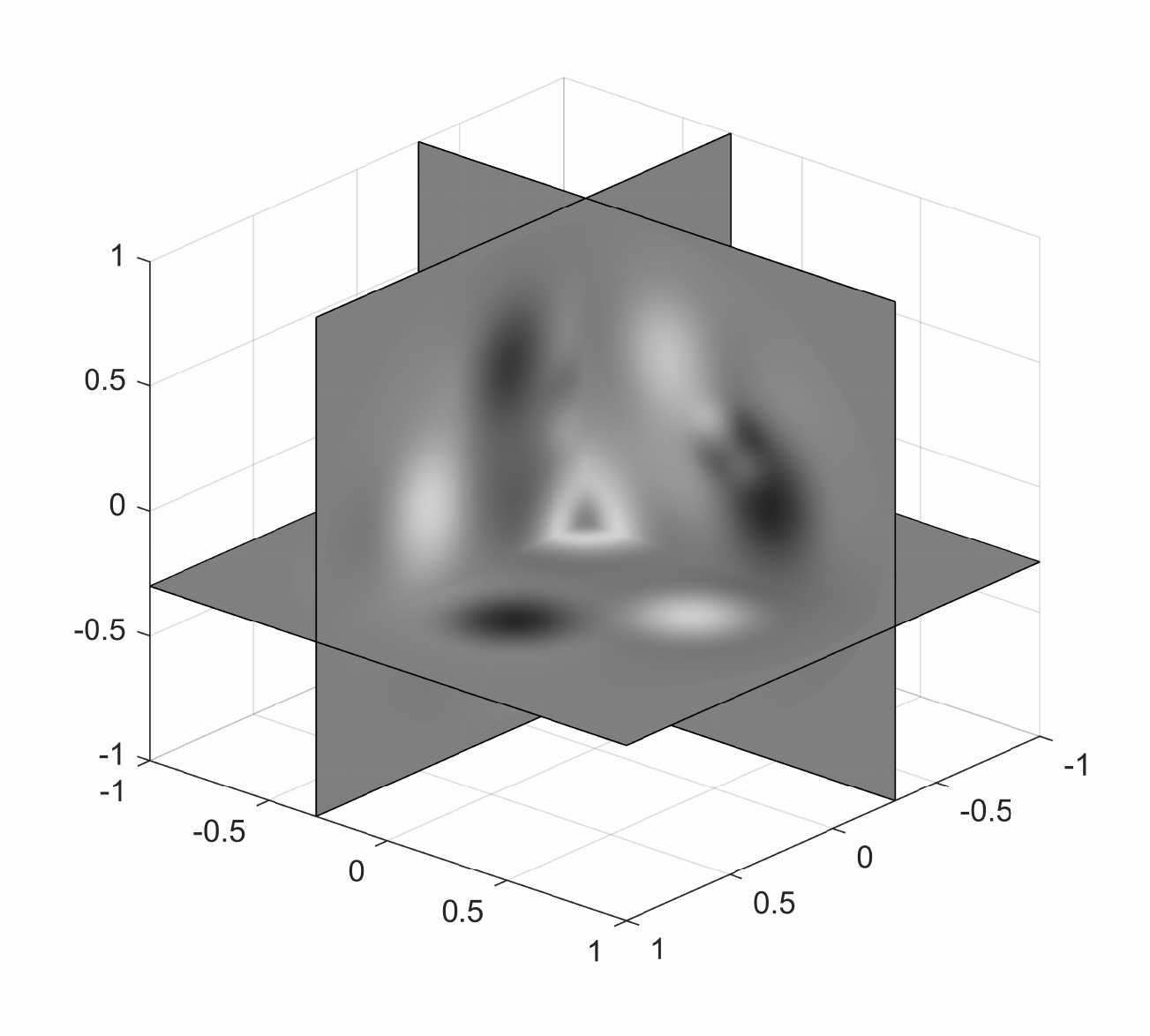}  \\
$F^\mathrm{s}_1(x)$ &  $F^\mathrm{s}_2(x)$  & $F^\mathrm{s}_3(x)$ \\
\\
\includegraphics[scale=0.48]{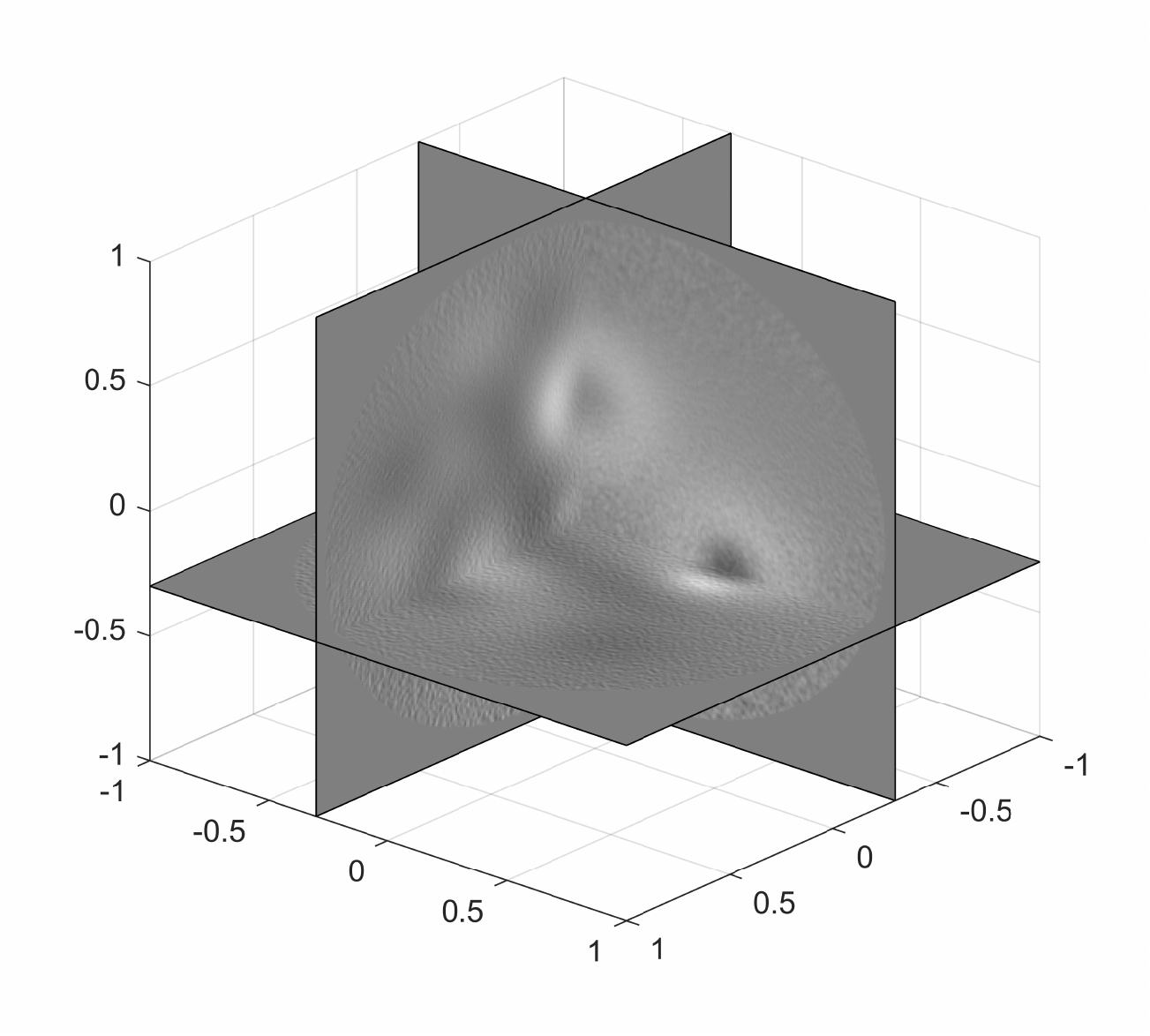} \hspace{-10mm}&
\hspace{-10mm} \includegraphics[scale=0.48]{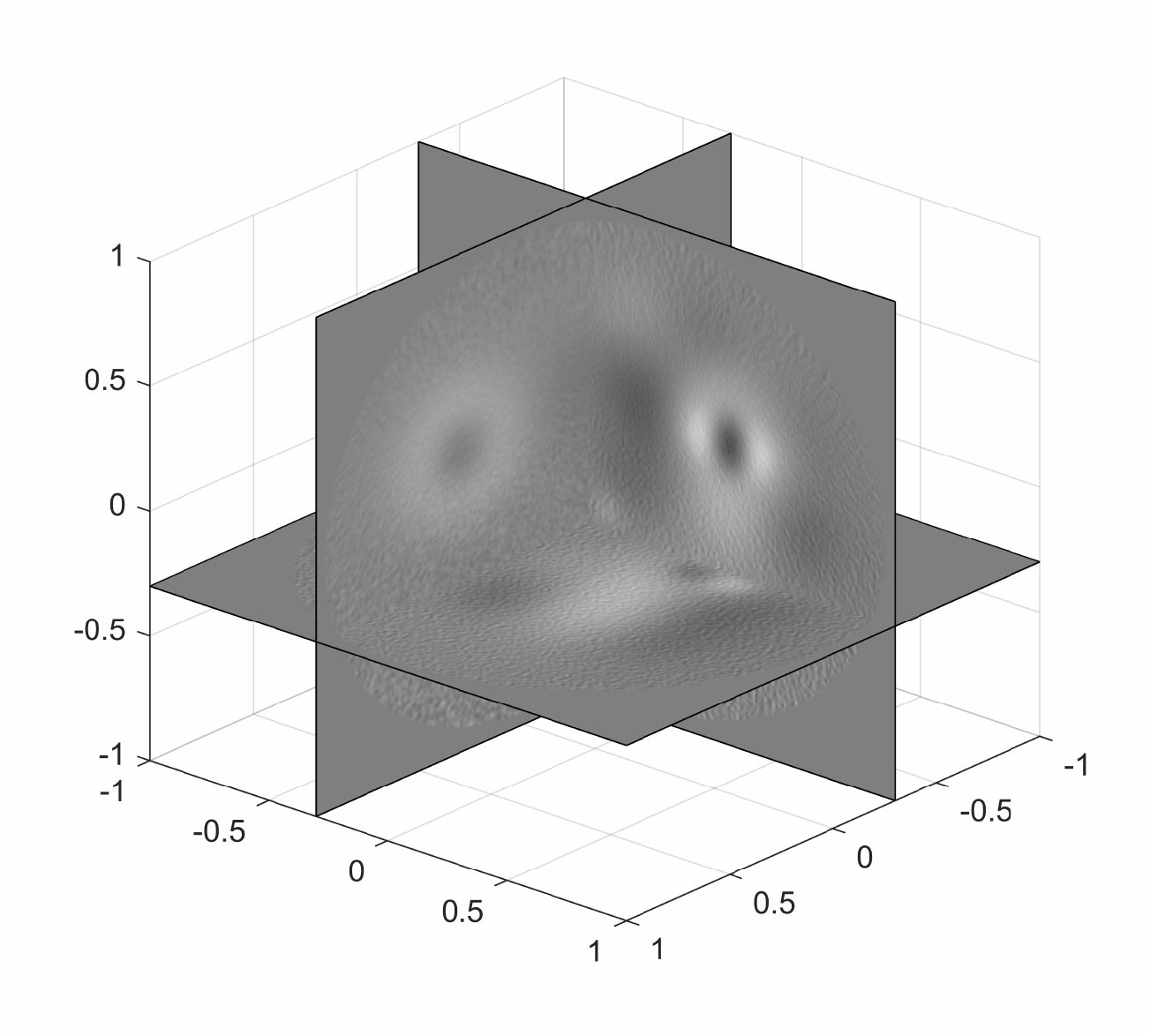} \hspace{-10mm}&
\hspace{-10mm} \includegraphics[scale=0.48]{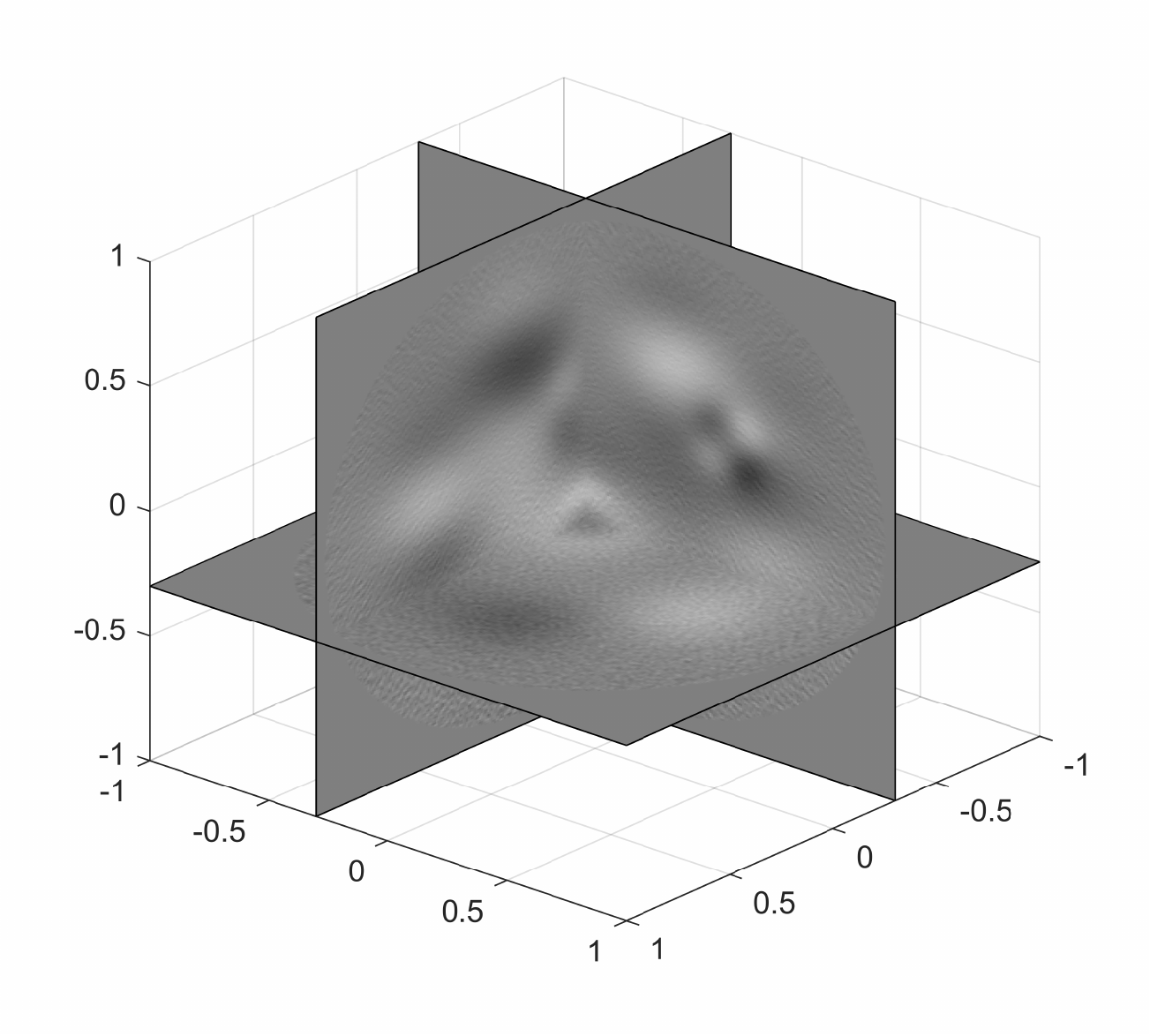}  \\
$F^\mathrm{p}_1(x)$ &  $F^\mathrm{p}_2(x)$  & $F^\mathrm{p}_3(x)$ \\
\\

\end{tabular}
\par
\caption{Solenoidal and potential parts of the field,  $F^\mathrm{s}$ and  $F^\mathrm{p}$, reconstructed from noisy data}
\label{F:rec_noise}%
\end{figure}

In order to test this conclusion we added to the data $\mathcal{D}%
_{1}^{\shortparallel}F$, $\mathcal{D}_{2}^{\shortparallel}F$, and
$\mathcal{W}_{1}^{\shortparallel}F$ a small normally distributed spatially
uncorrelated noise with relative intensity $0.1$\% in $L_{2}$ norm. Spectral
differentiation in $p$ in (\ref{E:invert_for_pot}) and in (\ref{E:3FBP}) was
replaced by the standard second order symmetric finite difference formula.
This has a mild regularizing effect compared with the spectral
differentiation. The fields $F^{\mathrm{s}}(x)$ and $F^{\mathrm{p}}(x)$
reconstructed from the noisy data are shown in Figure~\ref{F:rec_noise}
(the gray scale used in this figure is the same as in Figure~\ref{F:phantom}).
Comparison with the
Figure~\ref{F:rec_nonoise} shows that the solenoidal part $F^{\mathrm{s}}(x)$ is little
affected by this mild noise, while reconstructed $F^{\mathrm{p}}(x)$ contains
much stronger high frequency artifacts (the reader may want to magnify the
figure to see this clearly). Indeed, a quantitative comparison reveals that
the relative error in $F^{\mathrm{s}}(x)$ is 1.1\% in $L_{2}$ norm and 1.3\%
in $L_{\infty}$ norm. On the other hand, the relative error in $F^{\mathrm{p}%
}(x)$ is 63\% in $L_{2}$ norm and 74\% in $L_{\infty}$ norm.

\begin{figure}[t]
\noindent \begin{tabular}{ccc}
\includegraphics[scale=0.48]{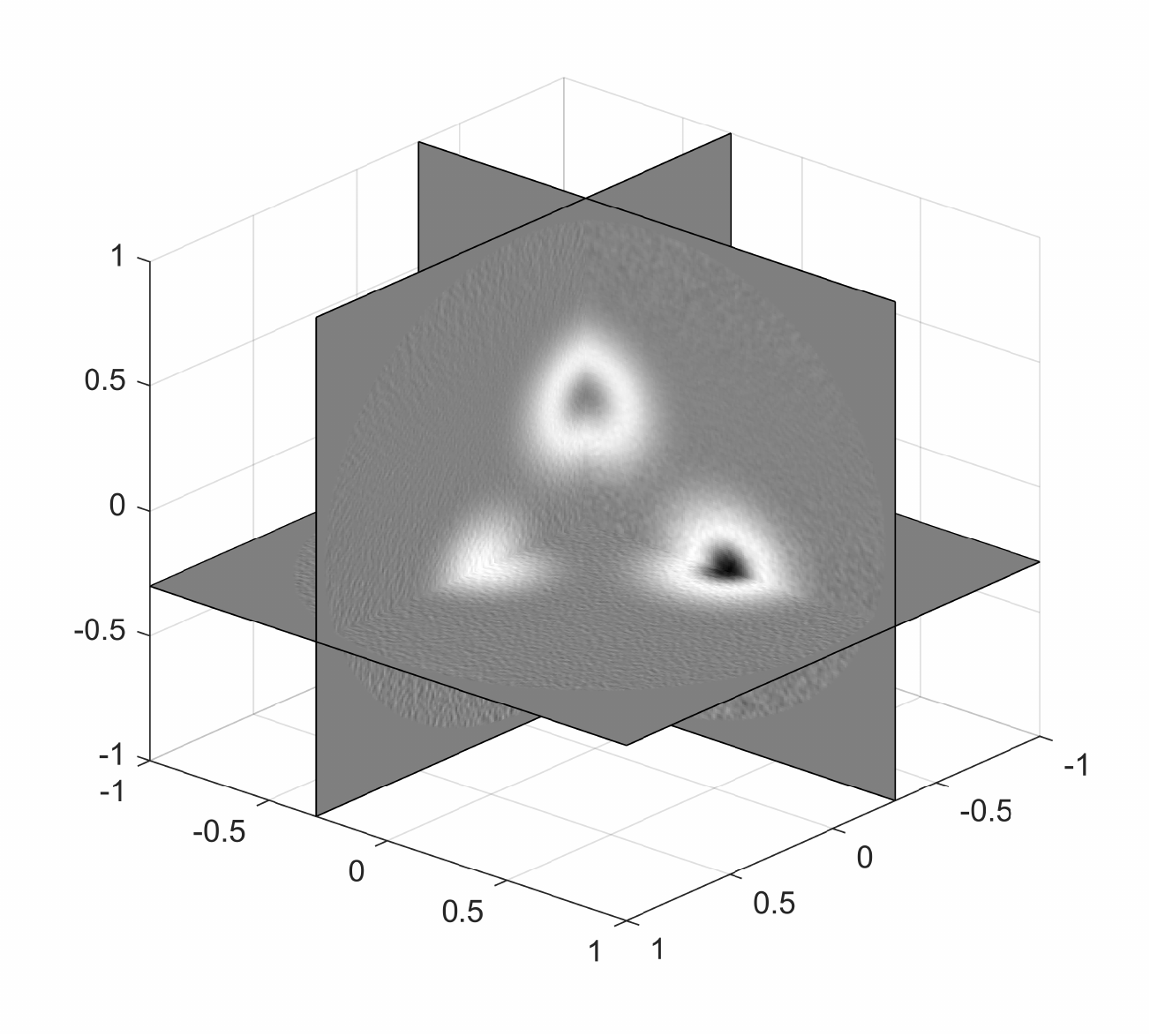} \hspace{-10mm}&
\hspace{-10mm} \includegraphics[scale=0.48]{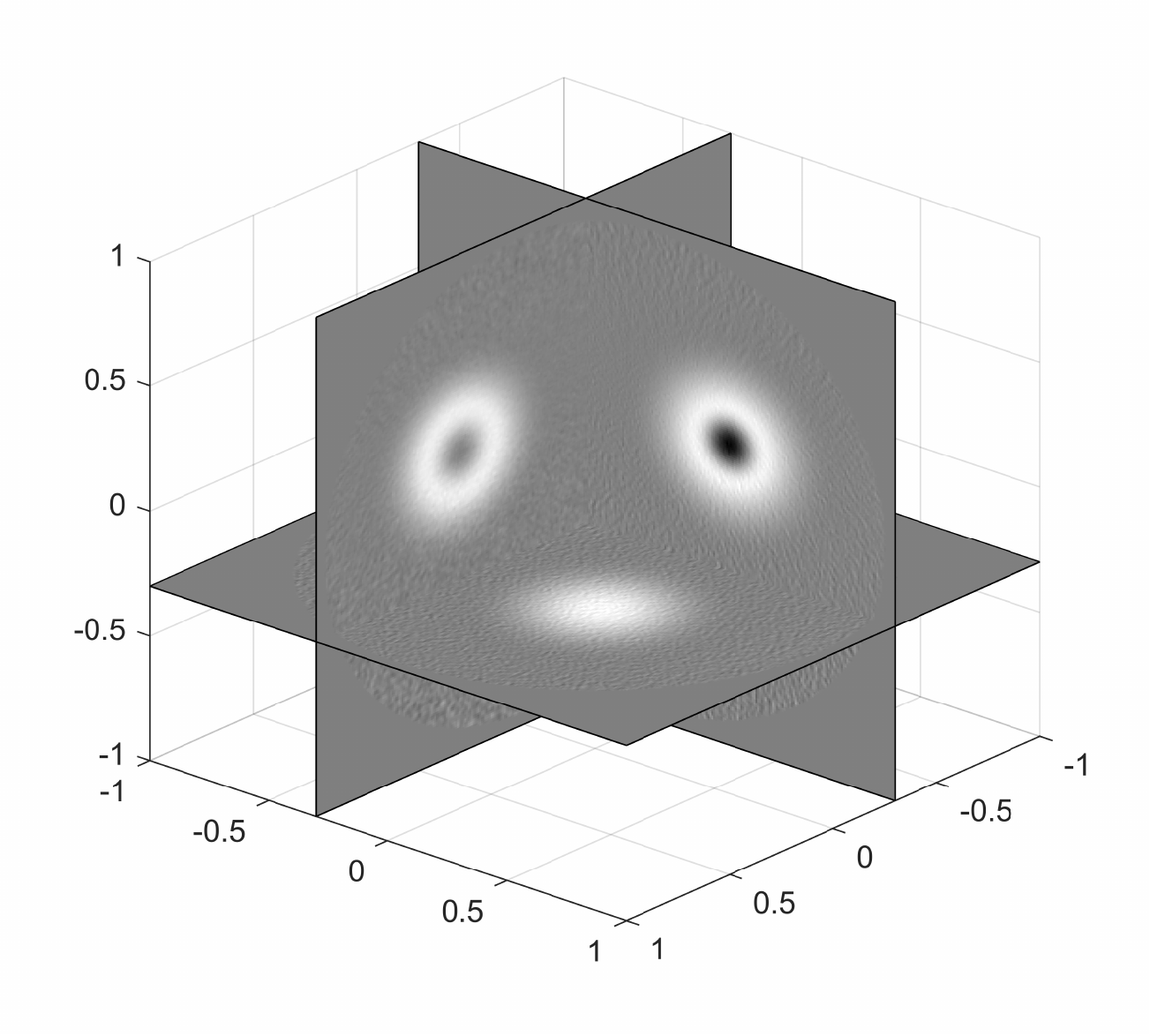} \hspace{-10mm}&
\hspace{-10mm} \includegraphics[scale=0.48]{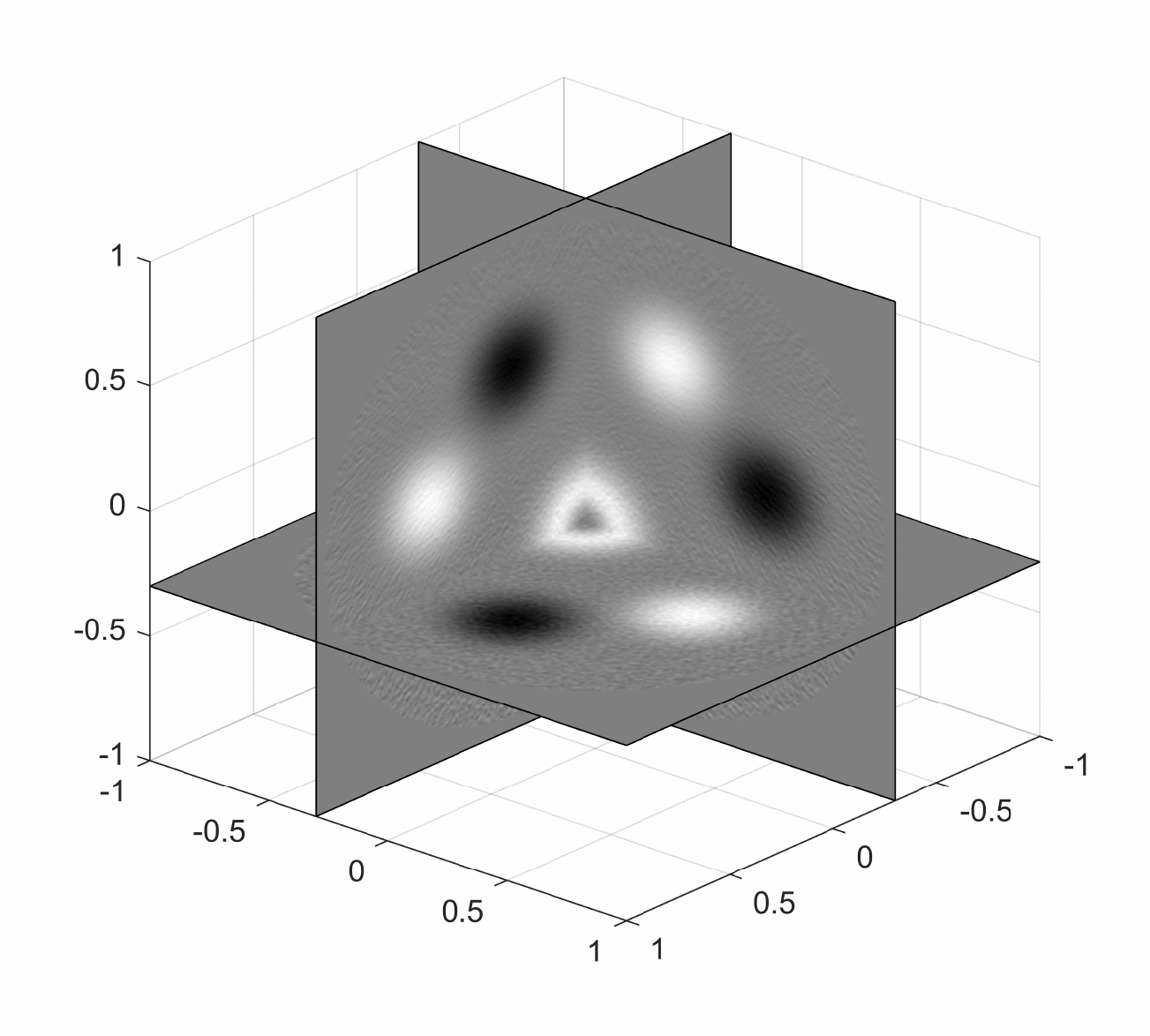}  \\
$F^\mathrm{s}_1(x)$ &  $F^\mathrm{s}_2(x)$  & $F^\mathrm{s}_3(x)$ \\
\end{tabular}
\par
\caption{Field $F$ reconstructed from noisy data}
\label{F:rec_noise_total}%
\end{figure}

{
The total reconstructed field is the sum of $F^{\mathrm{s}}(x)$ and
$F^{\mathrm{p}}(x).$ It is depicted in Figure~\ref{F:rec_noise_total} (the
gray scale is the same as in Figure~\ref{F:phantom}). Due to the high level of
artifacts in $F^{\mathrm{p}}(x),$ the total field also contains significant
error, with the relative error equal to 36\% in $L_{2}$ norm and 41\% in
$L_{\infty}$ norm. It should be noted that the high error in $F^{\mathrm{p}%
}(x)$ is a manifestation of the poor conditioning of the problem of
reconstructing the potential part of the field from a linearly weighted
longitudinal transform $\mathcal{W}_{1}^{\shortparallel}F$. Indeed, formula
(\ref{E:Radon_of_solenoidal}) shows that the Radon transform of $F^{\mathrm{s}%
}$ is expressed as a linear combination of data $\mathcal{D}_{2}%
^{\shortparallel}F$. Thus, the conditioning of finding $F^{\mathrm{s}}$ is
similar to conditioning of inverting the standard scalar Radon transform. On
the other hand, in the equation (\ref{E:usingW_4}) the Radon transform
$\mathcal{R}\left(  F_{k}^{\mathrm{p}}\right)  $ is expressed through the
derivative of the data $\frac{\partial}{\partial p}\,\mathcal{W}%
_{1}^{\shortparallel}\left(  F\right)  $. This additional differentiation of
data makes the problem of reconstructing $F^{\mathrm{p}}$ significantly more
ill-posed than that of inverting the regular Radon transform. This leads to
the appearance of strong high frequency artifacts in the reconstructed
$F^{\mathrm{p}}$.

In order to convince the reader that this is indeed a high-frequency
phenomenon, we applied a low-pass linear filter to the total reconstructed
field $F(x),$ obtaining a smoothed field $F^{\mathrm{smooth}}(x)$. In detail,
each component $F_{k}^{\mathrm{smooth}}(x)$ of $F^{\mathrm{smooth}}(x)$ was
obtained by applying filter $\eta(\xi)$ in the Fourier domain:
\[
F_{k}^{\mathrm{smooth}}(x)=\mathcal{F}^{-1}[\eta(\xi)[\mathcal{F(}F_{k}%
)](\xi)](x),\quad k=1,2,3.
\]
where $\mathcal{F}$ and $\mathcal{F}^{-1}$ are the forward and inverse Fourier
transforms, and filter $\eta(\xi)$ was given by the formula%
\[
\eta(\xi)=0.5\left(  1+\cos\frac{\pi|\xi|}{0.4f^{\mathrm{Nyquist}}}\right)
\text{ for }|\xi|<0.4f^{\mathrm{Nyquist}},\text{ 0 otherwise, }%
\]
where $f^{\mathrm{Nyquist}}$ is the Nyquist frequency of the spatial discretization in $x$.
The relative errors in the so found approximation $F^{\mathrm{smooth}}(x)$ where 12\% in $L_{2}$
norm and 19\% in $L_{\infty}$ norm.

We would like to stress that the reconstruction algorithm presented here,
based on direct discretization of our inversion formulas, is meant only to
illustrate the exactness of these formulas (when applied to accurate data),
and to demonstrate the increased sensitivity of these formulas to noise (in
comparison to the standard Radon inversion). The development of a more
practical, efficient and robust algorithm is a matter of the future work. Such
an algorithm would require a prudent choice of a regularization technique, to
reduce the noise sensitivity. An optimal choice of such technique depends
heavily on the parameters of a particular application, such as the
signal-to-noise ratio, spectral content of the noise, desired resolution, etc.
For a general overview of classical regularization methods we refer the reader
to the book~\cite{tikhonov} and article~\cite{louis1996approximate}. The
regularization methods used recently in vector tomography include the singular
value decomposition~\cite{derevtsov2011singular}, the method of approximate
inverse~\cite{derevtsov2017numerical}, and an expansion in a series of
orthogonal polynomials~\cite{polya2015}. These topics, however, are outside of
the scope of the present paper. }

\section*{Appendix}

In the present Appendix we prove Theorem \ref{T:my_Helm} that establishes the
rates of decay at infinity of the potential and solenoidal parts of the field,
as given by equations (\ref{E:thm_phi})-(\ref{E:thm_second_der}).

We will need the following Lemma.

\begin{lemma}
\label{T:convlemma}Consider convolution $h$ of functions $f$ and $g$ defined
as follows%
\begin{equation}
h(x)=\int\limits_{\mathbb{R}^{d}}f(y)g(x-y)dy,\qquad x\in\mathbb{R}^{d},
\label{E:gen_conv}%
\end{equation}
If $f(x)$ and $g(x)$ are locally integrable and satisfy the inequalities:%
\begin{equation}
|f(x)|\leq\frac{C_{f}}{(1+|x|)^{K}},\qquad|g(x)|\leq\frac{C_{g}}{(1+|x|)^{M}%
},\qquad K>0,\qquad M\geq d+K, \label{E:lemma_conditions}%
\end{equation}
then there is a constant $C$ such that convolution $h(x)$ is bounded as
follows:%
\[
|h(x)|\leq\frac{C}{(1+|x|)^{K}}.
\]

\end{lemma}

\begin{proof}
Note that inequalities (\ref{E:lemma_conditions})\ imply that $g$ is
absolutely integrable over $\mathbb{R}^{d}$:%
\[
\int\limits_{\mathbb{R}^{d}}|g(y)|dy=A<\infty.
\]
For a fixed $x$, split the integral (\ref{E:gen_conv}) as follows:%
\[
h(x)=I_{B}(x)+I_{O}(x),\qquad I_{B}(x)\equiv\int\limits_{ B(R)}
f(y)g(x-y)dy,\qquad I_{O}(x)\equiv\int\limits_{\mathbb{R}^{d}\backslash
B(R)}f(y)g(x-y)dy,
\]
where $B(R)$ is a ball of radius $R=|x|/2$ centered at the origin. Note that
the volume $|B(R)|$ of the ball is
\[
|B(R)|=C_{d}R^{d}=2^{-d}C_{d}|x|^{d},
\]
where $C_{d}$ is the volume of the unit ball in $\mathbb{R}^{d}.$ Obviously,
for any $y\in B(R),$ $|y|\leq R.$ Since $|x|=2R,$ $|x-y|\geq R,$ and%
\[
|g(x-y)|\leq\frac{C_{g}}{(1+|x-y|)^{M}}    \leq\frac{C_{g}}{(1+R)^{M}}%
=\frac{2^{M}C_{g}}{(2+|x|)^{M}}.
\]
Therefore, $I_{B}$ can be bounded as follows%
\begin{equation}
|I_{B}|\leq\frac{C_{g}C_{f}|B(R)|}{(1+R)^{M}}=\frac{C_{g}C_{f}C_{d}|x|^{d}%
}{2^{d-M}(2+|x|)^{M}}\leq\frac{C_{g}C_{f}C_{d}(1+|x|)^{d}}{2^{d-M}(1+|x|)^{M}%
}\leq\frac{C_{g}C_{f}C_{d}}{2^{d-M}(1+|x|)^{K}}. \label{E:IB_ineq}%
\end{equation}
On the other hand, for $y\in\mathbb{R}^{d}\backslash B(R),$
$|f(y)|$ can be bounded by $\frac{C_{f}}{(1+R)^{K}}$ so that the following
inequality holds%
\begin{equation}
|I_{O}|\leq\int\limits_{\mathbb{R}^{d}\backslash B(R)}|f(y)||g(x-y)|dy\leq
\frac{C_{f}}{(1+R)^{K}}\int\limits_{\mathbb{R}^{d}}|g(y)|dy\leq A\frac
{2^{K}C_{f}}{(2+|x|)^{K}}\leq A\frac{2^{K}C_{f}}{(1+|x|)^{K}}.
\label{E:IO_ineq}%
\end{equation}
Finally, by combining inequalities (\ref{E:IB_ineq}) and (\ref{E:IO_ineq}) one
proves Lemma{~\ref{T:convlemma}}.
\end{proof}

We are ready to prove Theorem \ref{T:my_Helm}.

\begin{proof}
First, we establish the rate of decay at infinity of the potential
$\varphi(x)$ given by the convolution~(\ref{E:potential_conv}). We note that
divergence $\Phi(x)$ belongs to the Schwartz space $\mathcal{S}(\mathbb{R}%
^{d})$ and, therefore, for any $l\geq0,$ there is a constant $C_{l}$ such that
$|\Phi(x)|\leq C_{l}/(1+|x|^{l}).$ On the other hand, the derivatives of the
fundamental solution $G(x)$ decay as follows:%
\begin{equation}
|D^{\alpha}G(x)|=\mathcal{O}\left(  \frac{1}{|x|^{d+  {|\alpha|}   -2}}\right)
,\qquad|\alpha|= 1,2,3,4. \label{E:decay_of_G}%
\end{equation}
Let us introduce an infinitely smooth nonnegative cut-off function $\eta(t),$
$t\in\mathbb{R}$, with $\eta(t)=1$ for every $t\in(-1/2,1/2)$ and $\eta(t)=0$
for $|t|\geq1.$ Convolution (\ref{E:potential_conv}) can be re-written as%
\begin{align*}
\varphi(x)  &  =I_{1}(x)+I_{2}(x),\\
I_{1}(x)  &  \equiv\int\limits_{|x-y|<1}\Phi(y)G(x-y)\eta(|x-y|)dy,\quad
I_{2}(x)\equiv\int\limits_{\mathbb{R}^{d}}\Phi(y)G(x-y)(1-\eta(|x-y|))dy,
\end{align*}
The first term $I_{1}(x)$ can be bounded as%
\[
\left\vert I_{1}(x)\right\vert =\left\vert \int\limits_{|x-y|<1}%
\Phi(y)G(x-y)\eta(|x-y|)dy\right\vert =\left\vert \int\limits_{|u|<1}%
\Phi(x-u)G(u)\eta(|u|)du\right\vert \leq C_{G}\max_{|u|<1}|\Phi(x-u)|,
\]
where
\begin{equation}
C_{G}\equiv\int\limits_{|u|<1}|G(u)|\eta(|u|)du. \label{E:CG}%
\end{equation}
Then $ {|}  I_{1}(x)|$ is bounded by $C_{G}C_{l}/(1+(|x|-1)^{l})$ for any $l\geq0.$

The second term is the following convolution
\begin{align*}
I_{2}(x)  &  =\int\limits_{\mathbb{R}^{d}}G(x-y)(1-\eta(|x-y|))\nabla_{y}\cdot
F(y)dy=\int\limits_{\mathbb{R}^{d}}\nabla_{x}\left[  G(x-y)(1-\eta
(|x-y|))\right]  \cdot F(y)dy\\
&  =\sum_{j=1}^{d}\int\limits_{\mathbb{R}^{d}}\frac{\partial}{\partial x_{j}%
}\left[  G(x-y)(1-\eta(|x-y|))\right]  F_{j}(y)dy
\end{align*}
The latter sum is the sum of convolutions of functions satisfying conditions
of Lemma \ref{T:convlemma}, where the role of $f$ is played by $F_{j}$ with
$M\geq2d-1$ (since $F_{j}$'s are Schwartz functions), and the role of $g$ is
played by $\frac{\partial}{\partial x_{j}}\left[  G(x)(1-\eta(|x|))\right]  $
with $K=d-1.$ Therefore, $I_{2}(x)$ has the desired rate of decay
$\mathcal{O}\left(  |x|^{1-d}\right)  .$ This term dominates the sum
$I_{1}(x)+I_{2}(x)$ at infinity. This proves equation (\ref{E:thm_phi}).

The estimate for the derivatives of $\varphi(x)$ can be obtained in a similar
way. Indeed%
\[
\frac{\partial}{\partial x_{j}}\varphi(x)=\frac{\partial}{\partial x_{j}}%
\int\limits_{\mathbb{R}^{d}}\Phi(y)G(x-y)dy=I_{3}(x)+I_{4}(x),
\]
where%
\[
I_{3}(x)\equiv\frac{\partial}{\partial x_{j}}\int\limits_{|x-y|<1}%
\Phi(y)G(x-y)\eta(|x-y|)dy,\qquad I_{4}(x)\equiv\frac{\partial}{\partial
x_{j}}\int\limits_{\mathbb{R}^{d}}\Phi(y)G(x-y)(1-\eta(|x-y|))dy.
\]
Now%
\begin{align*}
\left\vert I_{3}(x)\right\vert  &  =\left\vert \frac{\partial}{\partial x_{j}%
}\int\limits_{|u|<1}\Phi(x-u)G(u)\eta(|u|)du\right\vert =\left\vert
\int\limits_{|u|<1}\frac{\partial}{\partial x_{j}}\Phi(x-u)G(u)\eta
(|u|)du\right\vert \\
&  \leq C_{G}\max_{|u|<1}\left\vert \frac{\partial}{\partial x_{j}}%
\Phi(x-u)\right\vert =C_{G}\max_{|u|<1}\left\vert \frac{\partial}{\partial
x_{j}}\sum_{k=1}^{d}\frac{\partial}{\partial x_{k}}F_{k}(x-u)\right\vert ,
\end{align*}
where $C_{G}$ still given by (\ref{E:CG}). Since second derivatives of $F$ are
Schwartz functions, $|I_{3}(x)|$ decays faster than any power of $1/|x|.$ For
the term $I_{4}(x)\ $we observe:%
\begin{align*}
I_{4}(x)  &  \equiv\frac{\partial}{\partial x_{j}}\int\limits_{\mathbb{R}^{d}%
}\Phi(y)G(x-y)(1-\eta(|x-y|))dy=\frac{\partial}{\partial x_{j}}\sum_{j=k}%
^{d}\int\limits_{\mathbb{R}^{d}}\frac{\partial}{\partial x_{k}}\left[
G(x-y)(1-\eta(|x-y|))\right]  F_{k}(y)dy\\
&  =\sum_{j=k}^{d}\int\limits_{\mathbb{R}^{d}}\frac{\partial^{2}}{\partial
x_{k}\partial x_{j}}\left[  G(x-y)(1-\eta(|x-y|))\right]  F_{k}(y)dy.
\end{align*}
The rate of decay of derivatives $\frac{\partial^{2}}{\partial x_{k}\partial
x_{j}}\left[  G(x-y)(1-\eta(|x-y|))\right]  $ coincides with the decay rate of
$\frac{\partial^{2}}{\partial x_{k}\partial x_{j}}G(x);$ it is given by
(\ref{E:decay_of_G}). Now, the application of Lemma \ref{T:convlemma}
establishes that $|I_{4}(x)|=\mathcal{O}\left(  |x|^{-d}\right)  .$ This
proves (\ref{E:thm_fields}) for $F^{\mathrm{p}}.$ The similar estimate for
$F^{\mathrm{s}}$ comes from $F^{\mathrm{s}}(x)=F(x)-F^{\mathrm{p}}(x),$ where
the second term dominates at infinity.

Finally, equations (\ref{E:thm_derivatives}) and (\ref{E:thm_second_der}) are
proven similarly, by transferring the derivatives onto $G(x)$ and using
(\ref{E:decay_of_G}) with $|\alpha|=3$ and $|\alpha|=4$, combined with Lemma~{\ref{T:convlemma}}.
\end{proof}

\bigskip

\section*{Acknowledgments}

The first author acknowledges support by the NSF, through the award NSF/DMS
1814592. The third author was supported in part by NSF grant DMS-1937229
through the Data Driven Discovery RTG's summer REU program at the University
of Arizona.



\begin{thebibliography}{10}

\bibitem{norton-vector89}
Norton S J 1989 Tomographic reconstruction of 2-d vector fields: application to flow
  imaging \textit{ Geophysical Journal International} \textbf{97}(1) 161--168

\bibitem{norton-vector92}
Norton S J 1992 Unique tomographic reconstruction of vector fields using boundary
  data \textit{ IEEE Transactions on image processing} \textbf{1}(3) 406--412

\bibitem{strahlen-review}
Sparr G and Str{\aa}hl{\'e}n K 1999
Vector field tomography: an overview \textit{ IMA Volumes in Mathematics and its Applications; Computational
  Radiology and Imaging: Therapy and Diagnostic} 110

\bibitem{schuster-review}
Schuster T 2008
 20 years of imaging in vector field tomography: a review
\textit{ Mathematical Methods in Biomedical Imaging and Intensity-Modulated Radiation Therapy (IMRT)}, volume~7.
(CRM), Birkh{\"a}user

\bibitem{shar-book}
 Sharafutdinov V A 2012
\textit{ Integral geometry of tensor fields}, volume~1.
(Walter de Gruyter)

\bibitem{strahlen-expo}
Str{\aa}hl{\'e}n K 1997
Exponential vector field tomography
\textit{ International Conference on Image Analysis and Processing},
   (Springer) 348--355

\bibitem{bukh-kaz-vector}
Bukhgeim A A and Kazantsev S G 2003
Full reconstruction of a vector field from its attenuated vectorial Radon transform
In \textit{ Modelling, Identification and Control} 294--298

\bibitem{natt-vector}
Natterer F 2005
Inverting the attenuated vectorial Radon transform.
\textit{ J. Inverse Ill-posed Problems} \textbf{13}(1) 93--101

\bibitem{bal-atten}
Bal G 2004
On the attenuated Radon transform with full and partial measurements
\textit{ Inverse Problems} \textbf{20}(2) 399--418

\bibitem{shar-mom}
Krishnan V P,  Manna R, Sahoo S-K and  Sharafutdinov V A 2019
  Momentum ray transforms \textit{ Inverse Problems and Imaging} \textbf{13}(3) 679--701

\bibitem{mishra-weight}
Mishra R K  2020
  Full reconstruction of a vector field from restricted Doppler and
  first integral moment transforms in $\mathbb{R}^n$
  \textit{ Journal of Inverse and Ill-posed Problems} \textbf{28}(2) 173--184

\bibitem{polya2015}
Polyakova A 2015
  Reconstruction of a vector field in a ball from its normal Radon
  transform  \textit{ Journal of Mathematical Sciences} \textbf{205}(3) 418-439

\bibitem{polya2015num}
Polyakova A P and Svetov I E 2015
  Numerical solution of the problem of reconstructing a potential
  vector field in the unit ball from its normal Radon transform.
  \textit{ Journal of Applied and Industrial Mathematics} \textbf{9}(4) 547--558

\bibitem{NattBook} Natterer F 2001 \textit{The Mathematics of Computerized Tomography (Classics in Applied Mathematics)}
(Society for Industrial Mathematics) p 184

\bibitem{helgason2013Radon}
Helgason S 1999
  \textit{ The Radon Transform} Progress in Mathematics, volume~5
  (Springer)

\bibitem{Wen}
Wen H, Shah J and Balaban R S 1998
  Hall effect imaging.
  \textit{ IEEE transactions on biomedical engineering} \textbf{45}(1) 119--124

\bibitem{Grasland}
Grasland-Mongrain P,  Mari J-M,  Chapelon J-Y, and Lafon C 2013
  Lorentz force electrical impedance tomography
  \textit{ IRBM} \textbf{34}(4-5) 357--360

\bibitem{Roth}
Roth B J and Schalte K 2009
  Ultrasonically-induced Lorentz force tomography
  \textit{ Medical \& biological engineering \& computing} \textbf{47}(6) 573--577

\bibitem{Zengin}
Zengin R and Gen{\c{c}}er N G 2016
  Lorentz force electrical impedance tomography using magnetic field
  measurements \textit{ Physics in Medicine \& Biology} \textbf{61}(16) 5887

\bibitem{Montalibet}
Montalibet A, Jossinet J, Matias A, and Cathignol D 2001
  Electric current generated by ultrasonically induced Lorentz force in
  biological media
  \textit{ Medical and Biological Engineering and Computing} \textbf{39}(1) 15--20

\bibitem{Kun-MAET}
Kunyansky L 2012 A mathematical model and inversion procedure for
  magneto-acousto-electric tomography
  \textit{ Inverse problems}  \textbf{28}(3) 035002

\bibitem{Ammari2015}
Ammari H, Grasland-Mongrain P, Millien P, Seppecher L, and
   Seo J-K 2015
  A mathematical and numerical framework for ultrasonically-induced
  Lorentz force electrical impedance tomography
  \textit{ Journal de Math{\'e}matiques Pures et Appliqu{\'e}es},
  \textbf{103}(6) 1390--1409

\bibitem{Kuku}
Kuchment P and Kunyansky L 2008
  Mathematics of thermoacoustic tomography.
  \textit{ Euro. J. Appl. Math.} \textbf{19} 191--224

\bibitem{KuKuHand}
Kuchment P and Kunyansky L 2015
  Mathematics of photoacoustic and thermoacoustic tomography.
  textit{ Handbook of mathematical methods in imaging}
  (Springer, New York)  1117--1167

\bibitem{KIW-eng}
Kunyansky L, Ingram C P and Witte R S 2017
  Rotational magneto-acousto-electric tomography \textsc{(MAET)}:
  Theory and experimental validation
  \textit{ Physics in Medicine \& Biology}  \textbf{62}(8) 3025

\bibitem{Sun-rapid}
Sun T,  Hao P,  Chin C-T,  Deng D,  Chen T, Chen Y,
  Chen M,  Lin H,  Lu M  and  Gao Y 2021
  Rapid rotational magneto-acousto-electrical tomography with filtered
  back-projection algorithm based on plane waves.
  \textit{ Physics in Medicine \& Biology} \textbf{66}(9) 095002

\bibitem{Xia-MAET-lgus}
Xia H, Ding G and  Liu G 2019
  Magneto-acousto-electrical tomography with magnetic induction based
  on laser-generated ultrasound transducer
  \textit{ Journal of Medical Imaging and Health Informatics}
  \textbf{9}(1) 183--187

\bibitem{Xia-MAET-lgus-SPIE}
Ding G, Xia H, Li X, and Liu G 2018
  Experimental study of magneto-acousto-electrical tomography based on
  laser-generated ultrasound technology
  \textit{ Tenth International Conference on Information Optics and
  Photonics}, volume 10964, page 109646A. International Society for Optics and
  Photonics

\bibitem{tikhonov}
Tikhonov A N and  Arsenin V Ia 1997
  \textit{ Solutions of Ill-posed Problems}
  (Wiley)

\bibitem{louis1996approximate}
Louis A K 1996
  Approximate inverse for linear and some nonlinear problems
  \textit{ Inverse problems} \textbf{12}(2) 175--190

\bibitem{derevtsov2011singular}
Derevtsov E Yu, Efimov A V, Louis A K, and  Schuster T 2011
  Singular value decomposition and its application to numerical
  inversion for ray transforms in 2d vector tomography
  \textit{ Journal of Inverse \& Ill-Posed Problems} \textbf{19} 689-715

\bibitem{derevtsov2017numerical}
 Derevtsov E Yu, Louis A K, Maltseva S V, Polyakova A P, and Svetov I E 2017
  Numerical solvers based on the method of approximate inverse for 2D
  vector and 2-tensor tomography problems
  \textit{ Inverse Problems} \textbf{33}(12) 124001

\end{thebibliography}
\end{document}